\pdfoutput=1
\documentclass[a4paper,11pt]{amsart}
\usepackage[hmarginratio={1:1},vmarginratio={1:1},lmargin=80.0pt,tmargin=90.0pt]{geometry}
\usepackage{calligra}
\usepackage{fontenc}
\usepackage{extarrows}
\usepackage[ugly]{nicefrac}

\usepackage{mathtools}
\mathtoolsset{mathic}
\usepackage[final]{microtype}
\usepackage{multicol}
\usepackage{graphicx}
\usepackage{tabularx,tabu}
\usepackage{latexsym,exscale,enumitem,amsfonts,amssymb,mathtools}
\usepackage{amsmath,amsthm,amsfonts,amssymb,amscd,textcomp,bbm}
\usepackage{stmaryrd}
\SetSymbolFont{stmry}{bold}{U}{stmry}{m}{n}
\usepackage[normalem]{ulem}
\usepackage{thmtools}
\usepackage[backend=biber]{biblatex}
\usepackage{etex}
\usepackage{fancybox}
\usepackage[table]{xcolor}
\usepackage{mathrsfs}
\DeclareMathAlphabet{\mathscrbf}{OMS}{mdugm}{b}{n}

\usepackage{caption}
\usepackage[cal=stixfancy]{mathalpha}
\usepackage[all]{xy}
\SelectTips{cm}{}

\definecolor{myorange}{RGB}{225,127,0}
\definecolor{mygreen}{RGB}{0,225,0}
\definecolor{mypurple}{RGB}{128,0,128}
\definecolor{myred}{RGB}{255,0,0}
\definecolor{myblue}{RGB}{0,0,195}
\definecolor{myyellow}{RGB}{210,210,0}
\definecolor{mycream}{RGB}{200,200,150}

\definecolor{dummy}{RGB}{100,100,100}
\definecolor{mygray}{gray}{0.7}

\definecolor{orchid}{RGB}{143,40,194}
\definecolor{lava}{RGB}{207,16,32}

\definecolor{mydarkblue}{RGB}{10,10,170}

\definecolor{sea}{RGB}{46,139,87}
\definecolor{tomato}{RGB}{255,99,71}

\usepackage{tikz}
\tikzset{anchorbase/.style={baseline={([yshift=-0.5ex]current bounding box.center)}},
  tinynodes/.style={font=\tiny,text height=0.75ex,text depth=0.15ex},
  dstrand/.style={line width=1.5, color=black},
}
\usetikzlibrary{cd}
\usetikzlibrary{decorations}
\usetikzlibrary{decorations.markings}
\usetikzlibrary{decorations.pathreplacing}
\usetikzlibrary{decorations.pathmorphing}
\usetikzlibrary{shapes,positioning,matrix,calc}
\usetikzlibrary{shapes.callouts}
\usetikzlibrary{patterns}
\tikzstyle directed=[postaction={decorate,decoration={markings,
    mark=at position #1 with {\arrow[draw=black, line width=0.3mm]{>}}}}]
\tikzstyle rdirected=[postaction={decorate,decoration={markings,
    mark=at position #1 with {\arrow[draw=black, line width=0.3mm]{<}}}}]

\usepackage{comment}
\makeatletter
\newcommand{\setword}[2]{%
  \phantomsection
  #1\def\@currentlabel{\unexpanded{#1}}\label{#2}%
}
\makeatother

\newcommand{\placeholder}{\underline{\phantom{a}}}

\DeclareMathSymbol{\smin}{\mathbin}{AMSa}{"39}

\newcommand{\C}{\mathbb{C}}
\newcommand{\Z}{\mathbb{Z}}

\newcommand{\K}{\mathbb{K}}

\newcommand{\N}{\mathbb{N}}

\newcommand{\typeADE}{\mathsf{ADE}}
\newcommand{\typeADET}{\mathsf{ADET}}
\newcommand{\typeDT}{\mathsf{DT}}
\newcommand{\typeA}{\mathsf{A}}
\newcommand{\typeD}{\mathsf{D}}
\newcommand{\typeE}{\mathsf{E}}
\newcommand{\typeT}{\mathsf{T}}

\newcommand{\mstuff}[1]{\mathsf{#1}}
\newcommand{\varstuff}[1]{\mathtt{#1}}
\newcommand{\algstuff}[1]{\mathrm{#1}}
\newcommand{\catstuff}[1]{\mathcal{#1}}
\newcommand{\twocatstuff}[1]{\mathscrbf{#1}}

\newcommand{\obstuff}[1]{\mathtt{#1}}
\newcommand{\morstuff}[1]{\mathrm{#1}}

\newcommand{\End}{\algstuff{E}\mathrm{nd}}
\newcommand{\Aut}{\algstuff{A}\mathrm{ut}}
\newcommand{\Sym}{\algstuff{S}\mathrm{ym}}
\newcommand{\Hom}{\algstuff{H}\mathrm{om}}

\newcommand{\oneHom}{\catstuff{H}\mathrm{om}}
\newcommand{\oneMod}{\catstuff{M}}

\newcommand{\acts}{\centerdot}
\newcommand{\Gros}{\mathrm{K}_0^{\oplus}}
\newcommand{\lb}{\lambda}

\newcommand{\iso}{\xrightarrow{\,\smash{\raisebox{-0.65ex}{\ensuremath{\scriptstyle\sim}}}\,}}

\newcommand{\zmatrix}{\mstuff{0}}
\newcommand{\imatrix}{\mstuff{I}}
\newcommand{\amatrix}[1][x]{\mstuff{A}_{#1}}
\newcommand{\batrix}[1][]{\mstuff{B}_{#1}}
\newcommand{\mmatrix}[1][x]{\mstuff{M}_{#1}}

\newcommand{\qpar}{\zeta}

\newcommand{\taft}{\algstuff{T}}
\newcommand{\gt}{\algstuff{g}}
\newcommand{\xx}{\algstuff{z}}

\newcommand{\tmod}{\catstuff{T}}
\newcommand{\stmod}{\underline{\catstuff{T}}}
\newcommand{\sstmod}{\overline{\catstuff{T}}}

\newcommand{\tmodule}[1][M]{\obstuff{#1}}

\newcommand{\tsimple}[1][r]{\obstuff{S}_{#1}}
\newcommand{\tinde}[1][{\ell,r}]{\obstuff{M}_{#1}}
\newcommand{\twmod}[1][a]{\obstuff{W}_{#1}}
\newcommand{\tprmod}[1][r]{\obstuff{P}_{#1}}

\newcommand{\cpoly}[1][e]{\mathrm{U}_{#1}}
\newcommand{\ctpoly}[1][e]{\tilde{\mathrm{U}}_{#1}}
\newcommand{\xpar}{\varstuff{X}}
\newcommand{\ypar}{\varstuff{Y}}
\newcommand{\zpar}{\varstuff{Z}}

\newcommand{\cell}[1][1]{\varstuff{C}_{#1}}

\newcommand{\JD}{\obstuff{J}_{\typeD}}
\newcommand{\JT}{\obstuff{J}_{\typeT}}
\newcommand{\J}{\obstuff{J}}

\newcommand{\taftp}{\chi(\tmod)}
\newcommand{\taftb}{\algstuff{B}}
\newcommand{\taftm}[1][]{\algstuff{N}_{#1}}

\newcommand{\taobj}[1][d{,}\lambda]{\obstuff{A}_{#1}}
\newcommand{\gat}{\obstuff{g}}
\newcommand{\zapar}{\obstuff{z}}

\newcommand{\xmodule}[1][M]{\obstuff{#1}_{\taobj[]}}

\newcommand{\tpmod}{\catstuff{P}\mathrm{r}(\catstuff{T})}
\newcommand{\tpid}{\catstuff{P}}

\newcommand{\sqgroup}{\algstuff{u}_{q}}

\theoremstyle{definition}
\newtheorem{theoremm}{Theorem}[section]

\declaretheorem[style=definition,name=Theorem,numberlike=theoremm]{theorem}

\declaretheorem[style=definition,name=Corollary,numberlike=theoremm]{corollary}
\declaretheorem[style=definition,name=Lemma,numberlike=theoremm]{lemma}
\declaretheorem[style=definition,name=Lemma,numberlike=theoremm]{lemmaqed}

\declaretheorem[style=definition,name=Proposition,numberlike=theoremm]{proposition}
\declaretheorem[style=definition,name=Proposition 1,numbered=no]{proposition1}
\declaretheorem[style=definition,name=Proposition 2,numbered=no]{proposition2}
\declaretheorem[style=definition,name=Theorem 3,numbered=no]{theorem3}
\declaretheorem[style=definition,name=Theorem 4,numbered=no]{theorem4}
\declaretheorem[style=definition,name=Proposition,numberlike=theoremm]{propositionqed}

\declaretheorem[style=definition,name=Example,numberlike=theorem]{example}
\declaretheorem[style=definition,name=Definition,numberlike=theorem]{definition}

\declaretheorem[style=definition,name=Remark,numberlike=theorem]{remark}

\def\notation#1#2#3{\rlap{\hyperref[#1]{{\color{orchid}#2}}}\hspace*{8.2mm} \hbox to 47mm{#3\hfill}}

\allowdisplaybreaks

\numberwithin{equation}{section}
\usepackage[hypertexnames=false]{hyperref}
\usepackage{cleveref,bookmark}
\hypersetup{
    pdftoolbar=true,
    pdfmenubar=true,
    pdffitwindow=false,
    pdfstartview={FitH},
    pdftitle={Title},
    pdfauthor={Names},
    pdfsubject={},
    pdfcreator={Names},
    pdfproducer={Names},
    pdfkeywords={},
    pdfnewwindow=true,
    colorlinks=true,
    linkcolor=mydarkblue,
    citecolor=teal,
    filecolor=magenta,
    urlcolor=orchid,
    linkbordercolor=lava,
    citebordercolor=teal,
    urlbordercolor=orchid,
    linktocpage=true
}
\let\fullref\autoref
\def\makeautorefname#1#2{\expandafter\def\csname#1autorefname\endcsname{#2}}
\makeautorefname{equation}{Equation}%
\makeautorefname{footnote}{footnote}%
\makeautorefname{item}{item}%
\makeautorefname{figure}{Figure}%
\makeautorefname{table}{Table}%
\makeautorefname{part}{Part}%
\makeautorefname{appendix}{Appendix}%
\makeautorefname{chapter}{Chapter}%
\makeautorefname{section}{Section}%
\makeautorefname{subsection}{Section}%
\makeautorefname{subsubsection}{Section}%
\makeautorefname{paragraph}{Paragraph}%
\makeautorefname{subparagraph}{Paragraph}%
\makeautorefname{theorem}{Theorem}%
\makeautorefname{theo}{Theorem}%
\makeautorefname{thm}{Theorem}%
\makeautorefname{addendum}{Addendum}%
\makeautorefname{addend}{Addendum}%
\makeautorefname{add}{Addendum}%
\makeautorefname{maintheorem}{Main theorem}%
\makeautorefname{mainthm}{Main theorem}%
\makeautorefname{corollary}{Corollary}%
\makeautorefname{corol}{Corollary}%
\makeautorefname{coro}{Corollary}%
\makeautorefname{cor}{Corollary}%
\makeautorefname{lemma}{Lemma}%
\makeautorefname{lemm}{Lemma}%
\makeautorefname{lem}{Lemma}%
\makeautorefname{sublemma}{Sublemma}%
\makeautorefname{sublem}{Sublemma}%
\makeautorefname{subl}{Sublemma}%
\makeautorefname{proposition}{Proposition}%
\makeautorefname{proposit}{Proposition}%
\makeautorefname{propos}{Proposition}%
\makeautorefname{propo}{Proposition}%
\makeautorefname{prop}{Proposition}%
\makeautorefname{property}{Property}
\makeautorefname{proper}{Property}
\makeautorefname{scholium}{Scholium}%
\makeautorefname{step}{Step}%
\makeautorefname{conjecture}{Conjecture}%
\makeautorefname{conject}{Conjecture}%
\makeautorefname{conj}{Conjecture}%
\makeautorefname{question}{Question}
\makeautorefname{questn}{Question}
\makeautorefname{quest}{Question}
\makeautorefname{ques}{Question}
\makeautorefname{qn}{Question}
\makeautorefname{definition}{Definition}%
\makeautorefname{defin}{Definition}%
\makeautorefname{defi}{Definition}%
\makeautorefname{def}{Definition}%
\makeautorefname{dfn}{Definition}%
\makeautorefname{notation}{Notation}
\makeautorefname{nota}{Notation}
\makeautorefname{notn}{Notation}
\makeautorefname{remark}{Remark}%
\makeautorefname{rema}{Remark}%
\makeautorefname{rem}{Remark}%
\makeautorefname{rmk}{Remark}%
\makeautorefname{rk}{Remark}%
\makeautorefname{remarks}{Remarks}%
\makeautorefname{rems}{Remarks}%
\makeautorefname{rmks}{Remarks}%
\makeautorefname{rks}{Remarks}%
\makeautorefname{example}{Example}%
\makeautorefname{examp}{Example}%
\makeautorefname{exmp}{Example}%
\makeautorefname{exam}{Example}%
\makeautorefname{exa}{Example}%
\makeautorefname{algorithm}{Algorith}%
\makeautorefname{algo}{Algorith}%
\makeautorefname{alg}{Algorith}%
\makeautorefname{axiom}{Axiom}%
\makeautorefname{axi}{Axiom}%
\makeautorefname{ax}{Axiom}%
\makeautorefname{case}{Case}%
\makeautorefname{claim}{Claim}%
\makeautorefname{clm}{Claim}%
\makeautorefname{assumption}{Assumption}%
\makeautorefname{assumpt}{Assumption}%
\makeautorefname{conclusion}{Conclusion}%
\makeautorefname{concl}{Conclusion}%
\makeautorefname{conc}{Conclusion}%
\makeautorefname{condition}{Condition}%
\makeautorefname{condit}{Condition}%
\makeautorefname{cond}{Condition}%
\makeautorefname{construction}{Construction}%
\makeautorefname{construct}{Construction}%
\makeautorefname{const}{Construction}%
\makeautorefname{cons}{Construction}%
\makeautorefname{criterion}{Criterion}%
\makeautorefname{criter}{Criterion}%
\makeautorefname{crit}{Criterion}%
\makeautorefname{exercise}{Exercise}%
\makeautorefname{exer}{Exercise}%
\makeautorefname{exe}{Exercise}%
\makeautorefname{problem}{Problem}%
\makeautorefname{problm}{Problem}%
\makeautorefname{probm}{Problem}%
\makeautorefname{prob}{Problem}%
\makeautorefname{solution}{Solution}%
\makeautorefname{soln}{Solution}%
\makeautorefname{sol}{Solution}%
\makeautorefname{summary}{Summary}%
\makeautorefname{summ}{Summary}%
\makeautorefname{sum}{Summary}%
\makeautorefname{operation}{Operation}%
\makeautorefname{oper}{Operation}%
\makeautorefname{observation}{Observation}%
\makeautorefname{observn}{Observation}%
\makeautorefname{obser}{Observation}%
\makeautorefname{obs}{Observation}%
\makeautorefname{ob}{Observation}%
\makeautorefname{convention}{Convention}%
\makeautorefname{convent}{Convention}%
\makeautorefname{conv}{Convention}%
\makeautorefname{cvn}{Convention}%
\makeautorefname{warning}{Warning}%
\makeautorefname{warn}{Warning}%
\makeautorefname{note}{Note}%
\makeautorefname{fact}{Fact}%

\makeautorefname{problem}{CP}%
\makeautorefname{gproblem}{GP}%
\makeautorefname{dihedral}{`The dihedral story'}%

\usepackage{fancyhdr}
\begin{document}
\vbadness=10001
\hbadness=10001
\title[$2$-representation theory of the Taft algebra]{$2$-representation theory of the Taft algebra}
\author{Raphaël Paegelow \\\today}
\address{Department of Mathematics, University of Kaiserslautern–Landau, Postfach 3049, 67663 Kaiserslautern, Germany}
\email{raphael.paegelow@rptu.de}

\maketitle

\begin{abstract}
We study the $2$-representation theory of the Taft category, the tensor category of finite-dimensional modules over the Taft algebra. This category has two left cells. We first classify suitable transitive modules over its based representation ring. For the upper cell, we recall a construction by Etingof and Ostrik giving an infinite family of pairwise nonequivalent simple $2$-representations. For the bottom cell, we pass to the semisimplification of the Taft category and identify a distinguished fusion subcategory with an adjoint $\mathfrak{sl}_2$ Verlinde category. Using algebra objects associated with $\typeADET$-diagrams and categorical Clifford theory, we classify the regular exact indecomposable module categories of this semisimplification. For these regular module categories, the split Grothendieck group is a complete invariant. Finally, we identify the two cell $2$-representations of the Taft category.
\end{abstract}

\section{Introduction}

This article is a first step towards a complete description of the categorical representation theory of the Taft algebra based on the cell theory of $2$-representations, developed in \cite{MM14, MM11}.

The Taft algebra, originally defined in \cite{Taft}, is a finite-dimensional non-semisimple algebra, which has the advantage of having an explicit but non-trivial representation theory. At the same time, it exhibits many features characteristic of nonsemisimple algebras. The Taft algebra is a Hopf algebra that can be identified with the Borel part of the small quantum group of $\mathfrak{sl}_2$ at roots of unity \cite[Lemma 8.4]{KW16}. Moreover, this algebra is not unimodular. The tensor category of finite-dimensional Taft modules, which we call the Taft category, has a unique nonzero proper thick tensor ideal generated by all projective modules \cite[Theorem 5.5]{PW15}. Another special feature of this algebra is that it has only one-dimensional simple modules. From the perspective of cell $2$-representation theory, the Taft category has only two cells, which we refer to as the bottom and the upper cell. The upper cell is mainly controlled by projective modules, whereas the bottom cell is associated with the stable category and its semisimplification. Thus, the Taft category provides a good example of a tensor category that is not a fusion category and whose categorical representation theory can be described concretely.

Take $n\in\Z_{\geq 2}$, and $\K$ an algebraically closed field  whose characteristic does not divide $n$. Fix a primitive $n^{\mathrm{th}}$ root of unity $\qpar\in\K$. By default, all categories and functors are assumed to be additive, $\K$-linear and all monoidal categories will be strict (thanks to Mac Lane's coherence theorem).

The main results of this article are the following. First, we classify certain transitive modules over the representation ring of the Taft algebra associated with the upper and bottom cells. These modules arise as decategorifications of the module categories considered later in the article.

\begin{proposition1}[{\fullref{proposition:taft-big-apex}}]
For each divisor $d$ of $n$, there exists an admissible, transitive, Taft-split $\N$-module with upper cell apex. Up to $\N$-equivalence, all such $\N$-modules are of this form.
\end{proposition1}

\begin{proposition2}[{\fullref{proposition:taft-small-apex}}]

For each $\typeADET$ diagram $\Gamma$ with Coxeter number equal to $n$ and $d$ a divisor of $n$, there exists an admissible, regular, transitive, NIM $\N$-module with bottom cell apex. Up to $\N$-equivalence, all such $\N$-modules are of this form.
\end{proposition2}

Let $\tmod:=\taft\text{-}\catstuff{M}_{\mathrm{fd}}$ denote the Taft category.

We obtain an explicit infinite family of non-equivalent simple $2$-representations. For each pair $(d,\lb) \in \Z_{\geq 1} \times \K$ with $d|n$, Etingof and Ostrik \cite[Section 4.3]{EO04} construct an algebra object. Denote by $\catstuff{M}_{d,\lb}$ its associated exact module category. We have shown the following theorem.

\begin{theorem3}[{\fullref{theorem:taft-module-cat}}]
The simple $2$-representations
$\catstuff{M}_{d,\lb}$ and $\catstuff{M}_{d^{\prime},\lambda^{\prime}}$
are equivalent as $2$-representations if and only if ${d=d^{\prime} \text{ and } \lambda=\lambda^{\prime}}$.
\end{theorem3}

The tensor category $\tmod$ is non-semisimple, hence Ocneanu rigidity \cite[\S 9.1]{EGNO} does not apply. Going one step further, we know that the Davydov-Yetter cohomology of the Taft category does not vanish in positive even degrees, cf. \cite[\S 6.2]{FGS24}. Ocneanu rigidity and Theorem $3$ motivate the following question. Let $\catstuff{C}$ be a fiat tensor category  over an algebraically closed field of characteristic $0$.

\medskip
\textbf{Question:} Is having finitely many equivalence classes of simple finitary $2$-representations equivalent to the vanishing of the Davydov-Yetter cohomology $\mathrm{H}^i(\morstuff{F})$ for all $i>0$ and every monoidal functor $\morstuff{F}:\catstuff{C} \to \catstuff{D}$, where $\catstuff{D}$ is a fiat tensor category?
\medskip

Finally, we give a complete classification of the regular, exact, indecomposable module categories of the semisimplification $\sstmod$ of the Taft category. It can be shown that these are exactly the simple, regular, finitary $2$-representations of the Taft category with bottom cell apex:

\begin{theorem4}[{\fullref{thm:class-bottom}}]
Taking the split Grothendieck group induces a bijection between
\begin{equation*}
\left\{ \begin{gathered} \text{ exact, indecomposable, regular }\\ \sstmod\text{-module categories} \end{gathered} \right\} \big/_{\simeq} \iso \left\{ \begin{gathered} \text{admissible, regular, transitive, NIM} \\  \N\text{-modules with bottom cell apex } \end{gathered} \right\} \big/_{\cong_{\N}}.
\end{equation*}
\end{theorem4}

In one sentence, decategorification loses essential information for the upper cell, whereas it is a complete invariant in the regular bottom cell setting.

The article is organized as follows. In the third section, we recall the representation theory of this non-semisimple algebra. Section 4 studies the representation ring of the Taft category and its representation theory. Sections 5 and 6 construct and study an infinite family of finitary, simple $2$-representations with upper cell apex. Section 7 classifies the regular, simple $2$-representations with bottom cell apex. Finally, section 8 computes the cell $2$-representations associated with the upper and bottom cells. These cell $2$-representations play a fundamental role in the theory of finitary $2$-representations, cf. \cite[Theorem 43]{MM11}, \cite[Theorem 8]{MM16c}, \cite[Theorem 18]{MM16c}.

\section{Acknowledgments}
The author is supported by the SFB-TRR 195 “Symbolic Tools in Mathematics and their Application” of the German Research Foundation (DFG). We would like to thank Ulrich Thiel for suggesting the study of $2$-representations of the Taft algebra with bottom cell apex and for useful comments and remarks. The paper originates from and further develops some ideas that were informally discussed
by Marco Mackaay, Volodymyr Mazorchuk,
Vanessa Miemietz, Daniel Tubbenhauer and Xiaoting Zhang in 2018. Parts of the preliminary calculations and the upper-cell analysis grew out of notes from these discussions. We are particularly grateful to Daniel Tubbenhauer for many fruitful discussions and for the useful correspondence.

\newpage
\section{Representation theory of the Taft algebra}\label{section:taft}

%%%%%%%%%%%%%%%%%%%%%%%%%%%%%%%%%%%%%%%%%%%%%%%%%%%

\subsection{The Taft algebra}\label{subsection:taft-def}

%%%%%%%%%%%%%%%%%%%%%%%%%%%%%%%%%%%%%%%%%%%%%%%%%%%

\begin{definition}\label{definition:taft-def}
Let $\K\langle\gt,\xx\rangle$ denote the associative, unital algebra
freely generated by $\gt,\xx$.
The Taft algebra, denoted by $\taft=\taft_{n}(\qpar)$, is the quotient
\begin{gather*}
\K\langle\gt,\xx\rangle/(\gt^{n}=1,\;\xx^{n}=0,\;\gt\xx=\qpar\xx\gt).
\end{gather*}
\end{definition}

One can easily check that the assignments
\begin{gather}\label{eq:taft-hopf}
\begin{gathered}
\Delta(\gt)=\gt\otimes\gt,\quad\Delta(\xx)=\xx\otimes 1+\gt\otimes\xx,
\\
\epsilon(\gt)=1,\quad\epsilon(\xx)=0,\quad S(\gt)=\gt^{-1},\quad S(\xx)=-\gt^{-1}\xx,
\end{gathered}
\end{gather}
give $\taft$ the structure of a Hopf algebra.

\begin{remark}\label{remark:taft-def}
The algebra $\taft$ is not semisimple since the ideal generated by $\xx$ is nilpotent.
Our definition of $\taft$ is taken from \cite[Section 4.3]{EO04}.
Note that the conventions in \cite{COZ14} are slightly
different from these. We give an explicit (Hopf) algebra isomorphism: ${H_n(\zeta) \iso \taft, g \mapsto \gt^{-1}, h \mapsto \xx\gt^{-1}}$.
\end{remark}

\begin{example}\label{example:taft-sweedler}
For $n=2$ the algebra $\taft$ is Sweedler's four-dimensional Hopf algebra; cf. \cite[Chapter 4]{Sweed69}.
\end{example}

If $i\leq j$ are two integers, let us denote by $\llbracket i,j \rrbracket := \{i, i+1,\dots, j\}$.
The following lemma is immediate:

\begin{lemmaqed}\label{lemma:taft-basis}
The sets
\[
\{
\gt^{i}\xx^{j}
\mid
i,j\in\llbracket0,n-1\rrbracket
\}
, \quad
\{
\xx^{j}\gt^{i}
\mid
i,j\in\llbracket0,n-1\rrbracket
\}
\]
are two bases of $\taft$, and $\mathrm{dim}_{\K}(\taft)=n^{2}$.
\end{lemmaqed}

Recall that $\tmod$ denotes the category of finite-dimensional left $\taft$-modules. Since $\taft$ is a finite-dimensional $\K$-algebra, $\tmod$ is a Krull-Schmidt category.
By \fullref{lemma:taft-basis}, $\tmod$ is a
locally finite abelian category, as defined here \cite[\S 1.8]{EGNO}. Since $\taft$ is a Hopf algebra, cf. \eqref{eq:taft-hopf}, the category $\tmod$ is monoidal.

%%%%%%%%%%%%%%%%%%%%%%%%%%%%%%%%%%%%%%%%%%%%%%%%%%%

\subsection{Simple and indecomposable modules}\label{subsection:taft-reps}

%%%%%%%%%%%%%%%%%%%%%%%%%%%%%%%%%%%%%%%%%%%%%%%%%%%

The category $\tmod$ is non-semisimple, but representation-finite, as we will now see. In this subsection, we recall the classification of simple and indecomposable modules obtained in \cite{Cib93, Gun03}. This classification is also summarized in \cite{COZ14}.

\begin{lemmaqed}\label{lemma:taft-simples}(\cite[Lemma 2.1]{COZ14})
For $r\in \Z/n\Z$, there exists a
one-dimensional simple $\taft$-module $\tsimple=\K\{m\}$ determined by
\[
\gt\acts m=\qpar^{r}m,
 \quad
\xx\acts m=0.
\]
Moreover, for each $\tmodule[V]$ simple $\taft$-module, there exists $r \in \Z/n\Z$ such that $\tmodule[V] \cong \tsimple$, as $\taft$-modules.
\end{lemmaqed}

\begin{lemmaqed}\label{lemma:taft-inde}(\cite[Lemmas 2.2 and 2.3(3), Theorem 2.5]{COZ14})
For $1\leq \ell \leq n$ and $r \in \Z/n\Z$ there exists an
indecomposable $\taft$-module $\tinde=\K\{m_{1},\dots,m_{\ell}\}$ determined by
\[
\gt\acts m_{k}=\qpar^{r{+}k{-}1}m_{k}
,\quad
\xx\acts m_{k}=
\begin{cases}
m_{k+1}, & \text{ for }1\leq k\leq \ell\smin 1,
\\
0, & \text{ for }k=\ell.
\end{cases}
\]

Moreover, for each $\tmodule[V]$ indecomposable $\taft$-module, there exists $(\ell,r) \in \llbracket 1, n \rrbracket \times \Z/n\Z$ such that $\tmodule[V] \cong \tinde$, as $\taft$-modules.
\end{lemmaqed}

Recall that the head of a module $M$ is its largest semisimple quotient $\mathrm{hd}(M):=M/\mathrm{rad}(M)$, where $\mathrm{rad}(M)$ denotes the Jacobson radical of $M$. Its socle, denoted $\mathrm{soc}(M)$, the sum of all simple submodules of $M$ i.e. the (unique) maximal semisimple submodule of $M$.

\begin{remark}\label{remark:taft-socle-head}
One has $\mathrm{hd}(\tinde)\cong\tsimple$ and
$\mathrm{soc}(\tinde)\cong\tsimple[r{+}\ell{-}1]$. The module  $\tinde$ is simple if
and only if $\ell=1$, i.e. $\tinde[1,r]\cong\tsimple$.
\end{remark}

\begin{lemmaqed}\label{lemma:taft-prinjective}(\cite[Corollary 2.4]{COZ14})
The module $\tinde$ is projective, respectively injective, if
and only if $\ell=n$. Moreover, $\tprmod[r]:=\tinde[n,r]$ is the projective cover of $\tsimple$.
\end{lemmaqed}

\begin{remark}\label{remark:taft-simple-proj-inde}
From now on, the notation $\tinde$ will be reserved for $2\leq \ell\leq n-1$. When $\ell=1$ and $\ell=n$, we will use
$\tsimple$ or $\tprmod$.
\end{remark}

Let us recall the decomposition formulas for tensor products.

\begin{propositionqed}\label{proposition:taft-tensor}(\cite[Proposition 3.2]{COZ14})
Tensor products of indecomposable objects in $\tmod$ decompose as follows.
\begin{gather}\label{eq:tensordecomp-0}
\begin{gathered}
\tsimple\otimes\tsimple[r^{\prime}]\cong\tsimple[r^{\prime}]\otimes\tsimple\cong
\tsimple[r{+}r^{\prime}],
\\
\tsimple\otimes\tinde[\ell,r^{\prime}]\cong\tinde[\ell,r^{\prime}]\otimes\tsimple\cong
\tinde[\ell,r{+}r^{\prime}],
\\
\tsimple\otimes\tprmod[r^{\prime}]\cong\tprmod[r^{\prime}]\otimes\tsimple\cong
\tprmod[r{+}r^{\prime}],
\end{gathered}
\end{gather}
and hence, $\tsimple[0]$ is the tensor identity in
$\tmod$.

\begin{gather}\label{eq:tensordecomp}
\begin{gathered}
\tinde\otimes\tinde[\ell^{\prime},r^{\prime}]\cong\tinde[\ell^{\prime},r^{\prime}]\otimes\tinde\cong
\begin{cases}
{\textstyle\bigoplus_{i=1}^{\min(\ell,\ell^{\prime})}}\,\tinde[|\ell{-}\ell^{\prime}|{+}2i{-}1,r{+}r^{\prime}{-}i{+}\min(\ell,\ell^{\prime})],
& \text{if }\ell+\ell^{\prime}\leq n,
\\[6pt]
\begin{gathered}
{\textstyle\bigoplus_{i=1}^{n{-}\max(\ell,\ell^{\prime})}}\,\tinde[|\ell{-}\ell^{\prime}|{+}2i{-}1,r{+}r^{\prime}{-}i{+}\min(\ell,\ell^{\prime})]
\\[3pt]
\oplus
\text{ } {\textstyle\bigoplus_{i=1}^{\ell{+}\ell^{\prime}-n}}\,\tprmod[r{+}r^{\prime}{+}i{-}1],
\end{gathered}
& \text{if }\ell+\ell^{\prime}>n,
\end{cases}\\
\\
\tinde\otimes\tprmod[r^{\prime}]\cong\tprmod[r^{\prime}]\otimes\tinde\cong
{\textstyle\bigoplus_{i=1}^{\ell}}\,\tprmod[r{+}r^{\prime}{-}i{+}\ell],
\end{gathered}
\end{gather}
\begin{gather}\label{eq:tensordecomp-2}
\tprmod\otimes\tprmod[r^{\prime}]\cong\tprmod[r^{\prime}]\otimes\tprmod\cong
{\textstyle\bigoplus_{i=1}^{n}}\,\tprmod[r{+}r^{\prime}{-}i{+}n].
\end{gather}
In particular, combining \eqref{eq:tensordecomp-0},
\eqref{eq:tensordecomp} and \eqref{eq:tensordecomp-2},
the tensor product of any $\taft$-module
with any projective $\taft$-module is projective.
\end{propositionqed}

The following is an immediate consequence of
\fullref{proposition:taft-tensor}.

\begin{corollary}\label{corollary:taft-tensor-exact}
For any $r \in \Z/n\Z$, the functor
$\tsimple\otimes\placeholder\colon\tmod\to\tmod$ is an
invertible endofunctor on $\tmod$,
with inverse $\tsimple[n{-}r]\otimes\placeholder\colon\tmod\to\tmod$.
These functors are exact and send
projectives to projectives. A similar statement holds concerning the functor $\placeholder\otimes\tsimple\colon\tmod\to\tmod$.
\end{corollary}

%%%%%%%%%%%%%%%%%%%%%%%%%%%%%%%%%%%%%%%%%%%%%%%%%%%

\subsection{The monoidal category $\tmod$}\label{subsection:taft-mod-structure}

%%%%%%%%%%%%%%%%%%%%%%%%%%%%%%%%%%%%%%%%%%%%%%%%%%%

Since $\taft$ is a Hopf algebra, each object
$\tmodule$ has,
in particular, a left dual $\tmodule^{\ast}=\Hom_{\K}(\tmodule,\K)$ with
$\taft$-action given by
\[
(\algstuff{x}\acts\phi)(m)=\phi(S(\algstuff{x})m),\quad\algstuff{x}\in\taft,\; m\in\tmodule.
\]
The evaluation and coevaluation maps are the usual ones for
finite-dimensional vector spaces and their duals. These induce a left
dual $f^{\ast}\colon\tmodule[N]^{\ast}\to\tmodule^{\ast}$ for each morphism $f\colon\tmodule\to\tmodule[N]$.
Recall that: $(\tmodule\otimes\tmodule[N])^{\ast}\cong\tmodule[N]^{\ast}\otimes\tmodule^{\ast}$
and $(f\otimes g)^{\ast}=g^{\ast}\otimes f^{\ast}$,  cf. \cite[Exercise 2.10.7]{EGNO}.

The antipode is invertible, with inverse
$S^{-1}(\gt)=\gt^{-1}$ and $S^{-1}(\xx)=-\xx\gt^{-1}$. Therefore, $\tmodule$ also has a right dual
${}^{\ast}\tmodule$, with the same underlying vector space as $\tmodule^{\ast}$, but the $\taft$-action defined by
using the inverse of the antipode, i.e.
\begin{equation}
\label{eq:rigidity}
(\algstuff{x}\acts\phi)(m)=\phi(S^{-1}(\algstuff{x})m),\quad\algstuff{x}\in\taft,\; m\in\tmodule.
\end{equation}

Moreover, $S^{2}$ is an inner automorphism of $\taft$. Indeed,
$S^{2}(\algstuff{x})=\gt^{-1}\algstuff{x}\gt$, for all
$\algstuff{x}\in\taft$. Therefore, there
is a natural isomorphism $\tmodule^{\ast}\to{}^{\ast}\tmodule$ given by
\[
\phi\mapsto\gt\phi.
\]
This also implies that there are natural tensor isomorphisms
$\tmodule^{\ast\ast}\cong\tmodule\cong{}^{\ast\ast}\tmodule$ defining a pivotal structure on $\tmod$. To be more precise, for each $\tmodule \in \tmod$, consider the natural isomorphism
\begin{equation}\label{eq:piv-struct}
a_{\tmodule}:\begin{array}{ccc}
\tmodule &\iso& \tmodule^{\ast\ast}\\
 m &\mapsto& (f \mapsto f(\gt^{-1}m))
\end{array}.
\end{equation}

Duals of indecomposable modules are given by the next lemma.

\begin{lemmaqed}\label{lemma:taft-duals}
We have the dualities
\[
\tsimple^{\ast}\cong\tsimple[n{-}r]
,\quad
\tinde^{\ast}\cong\tinde[l,1{-}l{-}r]
,\quad
\tprmod^{\ast}\cong\tprmod[1{-}r]
\]
in the category $\tmod$.
\end{lemmaqed}

\makeautorefname{subsection}{Sections}

Recall from \cite[Definition 4.1.1]{EGNO}
that a tensor category is a locally finite abelian, rigid monoidal
category with the endomorphism space of the unit object being $\K$.
Moreover, a tensor category is called finite if its underlying abelian category is finite.
In the language of \cite{MM11} one calls a finite pivotal tensor category which is representation-finite
a fiat tensor category, and
\eqref{eq:rigidity} as well as \fullref{subsection:taft-def}
and \fullref{subsection:taft-reps} imply this proposition.

\makeautorefname{subsection}{Section}

\begin{propositionqed}\label{proposition:taft-fiat}
$\tmod$ is a non-semisimple fiat tensor category.
\end{propositionqed}

\begin{remark}\label{remark:taft-braided}
The category $\tmod$ is not braided for $n>2$, cf. \cite[Remark 3.13(3)]{COZ14},
although $\tmodule\otimes\tmodule[N]\cong\tmodule[N]\otimes\tmodule$
for all $\tmodule,\tmodule[N]\in\tmod$
by \fullref{proposition:taft-ring}. It is also not spherical for $n>2$.
Following \cite[Definition 4.7.11]{EGNO}, one can define the
quantum dimension $\qpar\dim$ of objects $\tmodule$
of $\tmod$. Using the pivotal structure \eqref{eq:piv-struct} we define
\[
\qpar\dim(\tmodule)=
\mathrm{tr}(a_{\tmodule})\in\End_{\tmod}(\tsimple[0]),
\]
where $\mathrm{tr}$ denotes the categorical trace.
One gets $\qpar\dim(\tsimple[1])=\qpar^{-1}\neq\qpar=\qpar\dim(\tsimple[\smin 1])$,
showing that $\tmod$ is not spherical as long as $n>2$.
\end{remark}

\begin{remark}\label{remark:taft-center}
Let $\K=\C$. As stated in \cite[Example 8.3.7]{EGNO} there
is a one-parameter family of quasi-triangular structures on $\taft_{2}$.

Moreover, the Drinfeld double $\algstuff{D}(\taft)$ of $\taft$ is a quasi-triangular
Hopf algebra for any $n\in\Z_{\geq 2}$.
According to \cite[Exercise 7.14.11]{EGNO}, for any odd integer $n$, there is an isomorphism of Hopf algebras
\[
\algstuff{D}(\taft)\cong\C[\Z/n\Z]\otimes_{\C}\sqgroup,
\]
where $\sqgroup=\sqgroup(\mathfrak{sl}_2)$ is Lusztig's small quantum group at the root of unity $q:=\qpar^{\frac{\smin 1}{2}}$ and the right hand side is equipped with a suitable Drinfeld twist of its standard coproduct.
This isomorphism implies that
\[
\algstuff{Z}(\tmod)\cong(\C[\Z/n\Z]\otimes_{\C}\sqgroup)\text{-}\catstuff{M}_{\mathrm{fd}},
\]
where $\algstuff{Z}(\tmod)$ denotes the Drinfeld center.
\end{remark}

%%%%%%%%%%%%%%%%%%%%%%%%%%%%%%%%%%%%%%%%%%%%%%%%%%%

\subsection{Categorical cells}\label{subsection:taft-cells}

%%%%%%%%%%%%%%%%%%%%%%%%%%%%%%%%%%%%%%%%%%%%%%%%%%%

From now on we will identify every monoidal category $\catstuff{C}$ with its delooping, i.e. with the $2$-category with a unique object $*$, hom-category $\mathrm{Hom}(*,*)=\catstuff{C}$ and horizontal composition given by the monoidal product in $\catstuff{C}$. Let us conclude this section by computing the cells of $\tmod$ in the sense of \cite{MM14, MM11}. We recall here the construction for the reader's convenience.

Let us write $\tmodule\subset_{\oplus}\tmodule[N]$ for two $\taft$-modules
if $\tmodule$ is a direct summand of $\tmodule[N]$.
The indecomposable objects $\tmodule$ of $\tmod$ give rise to the preorder $<$
by declaring
\[
\tmodule<\tmodule[M^{\prime}]
\;
\Leftrightarrow
\;
\text{there exists }\tmodule[M^{\prime\prime}]\text{ such that }
\tmodule[M^{\prime}]\subset_{\oplus}\tmodule[M]\otimes \tmodule[M^{\prime\prime}].
\]

\begin{remark}\label{rmk:cells}
The experienced reader may wonder which of the standard cell preorders is being used here (left, right, two-sided). In fact, there is no distinction in that case thanks to \fullref{lemma:taft-duals} and the fact that projective modules form a two-sided tensor ideal of $\tmod$.
\end{remark}

Cells of $\tmod$ are the
equivalence classes of the equivalence relation defined
on the set of indecomposable objects by
\[
\tmodule\sim\tmodule[M^{\prime}]
\;
\Leftrightarrow
\;
(\tmodule<\tmodule[M^{\prime}]\;\text{ and }\;\tmodule[M^{\prime}]<\tmodule).
\]

The following lemma is a direct consequence of \fullref{proposition:taft-tensor}.

\begin{lemma}\label{lemma:taft-cells}
	The category $\tmod$ has two cells, namely
\begin{gather*}
\cell[b]=\{\tsimple,\tinde\mid 2\leq \ell\leq n-1, r\in \Z/n\Z \}
\quad\ \text{ and }\quad
\cell[u]=\{\tprmod\mid r \in \Z/n\Z\},
\end{gather*}
of cardinalities $(n-1)n$ and $n$, respectively.
\end{lemma}
\begin{proof}
Since  $\tsimple[i] \otimes \tinde \simeq \tinde[\ell,r+i]$ and $\tsimple[i] \otimes \tprmod \simeq \tprmod[r+i]$, for a given $\ell$, all $\tinde$ belong to the same cell and all projective modules $\tprmod$ belong to the same cell. It is also clear that $\tsimple[0] < \tinde$ and since $\tsimple[0] \subset_{\oplus}\tinde \otimes \tinde[\ell,1 \smin \ell\smin r]$, we have $\tinde < \tsimple[0]$. Therefore all $\tsimple$ and $\tinde$ for $\ell \in \llbracket 2, n-1 \rrbracket$ form one cell and the projective covers form the second one. Finally, $\tsimple[0] < \tprmod$ but $\tprmod \nless \tsimple[0]$ since the tensor product of any module with a projective module is projective and thus cannot have $\tsimple[0]$ as direct summand since $\tsimple[0]$ is not projective.
\end{proof}

\begin{remark}\label{remark:taft-cell}
The two cells are ordered $\cell[b]<\cell[u]$. We will thus call $\cell[b]$ the bottom cell and $\cell[u]$ the upper cell.
\end{remark}

\section{Representation ring and decategorification}\label{section:taft-decat}
In this section, we recall the main result of \cite{COZ14} on the representation ring of $\tmod$ and then study important classes of modules for this ring, attached to each categorical cell.

\subsection{The representation ring of $\tmod$}\label{subsection:taft-repring}

%%%%%%%%%%%%%%%%%%%%%%%%%%%%%%%%%%%%%%%%%%%%%%%%%%%

Let us recursively define a family of two-variable Chebyshev polynomials
\begin{gather}\label{eq:chebyshev}
\begin{gathered}
\cpoly[1](\xpar,\ypar)=1,
\quad
\cpoly[2](\xpar,\ypar)=\xpar,
\\
\cpoly(\xpar,\ypar)=
\xpar\cdot\cpoly[e{-}1](\xpar,\ypar)
-
\ypar\cdot\cpoly[e{-}2](\xpar,\ypar),\;\text{for }e\geq 3.
\end{gathered}
\end{gather}

Denote by $\chi(\tmod)_{\Z}$ the representation ring of $\tmod$. As an abelian group, it is freely generated by the isomorphism classes $[\tmodule]$ of indecomposable finite-dimensional $\taft$-modules. Addition and multiplication are induced by direct sums and tensor products: ${[\tmodule]+[\tmodule[N]]:=[\tmodule \oplus \tmodule[N]]}$, ${[\tmodule][\tmodule[N]]:=[\tmodule \otimes \tmodule[N]]}$.
By \eqref{eq:tensordecomp-0} and \eqref{eq:tensordecomp-2},
we can restrict our attention to the indecomposable $\taft$-modules with $r=0$,
since tensoring by $\tsimple[1]$ increases $r$ by $1$. Moreover, for $\ell \in \llbracket 3,n \rrbracket$, equation
\eqref{eq:tensordecomp} gives
\begin{gather}\label{eq:taft-action-XY}
\tinde[2,0]\otimes\tinde[\ell{-}1,0]\cong\tinde[\ell,0]\oplus(\tsimple[1]\otimes\tinde[\ell{-}2,0])
,\quad
\tinde[2,0]\otimes\tprmod[0]\cong\tprmod[0]\oplus\tprmod[1].
\end{gather}
In $\chi(\tmod)_{\Z}$, we then have
\begin{equation*}
[\tinde[\ell,0]] = [\tinde[2,0]][\tinde[\ell{-}1,0]]-[\tsimple[1]][\tinde[\ell{-}2,0]], \quad \left([\tinde[2,0]]-[\tsimple[1]]-[\tsimple[0]]\right)[\tprmod[0]]=0.
\end{equation*}
This shows that $\left[\tinde[2,0]\right]$ and $\left[\tsimple[1]\right]$ generate $\chi(\tmod)_{\Z}$.

\begin{propositionqed}[{\cite[Theorem 3.10]{COZ14}}]\label{proposition:taft-ring}
We have an isomorphism of rings
\[
\begin{array}{ccc}
 \chi(\tmod)_{\Z} & \to & \Z[\xpar,\ypar]/(\ypar^{n}-1,(\xpar-\ypar-1)\cpoly[n](\xpar,\ypar))\\
 \left[\tinde[2,0]\right] & \mapsto & \xpar \\
 \left[\tsimple[1]\right] & \mapsto & \ypar
\end{array}
\]
where $\cpoly[n](\xpar,\ypar)$ is the two-variable Chebyshev polynomial from \eqref{eq:chebyshev}.
\end{propositionqed}

\begin{remark}
The ring $\chi(\tmod)_{\Z}$ is in general bigger than
the Grothendieck ring of $\tmod$, since the latter is just freely generated by
the isomorphism classes $[\tsimple]$, thus isomorphic to $\Z[\ypar]/(\ypar^{n}-1)$.
\end{remark}

Under the isomorphism from \fullref{proposition:taft-ring}
we have
\begin{gather}\label{eq:taft-apex}
[\tsimple]\mapsto\ypar^{r}
,\quad
[\tinde[\ell,0]]\mapsto\cpoly[\ell](\xpar,\ypar)
,\quad
[\tprmod[0]]\mapsto\cpoly[n](\xpar,\ypar).
\end{gather}

%%%%%%%%%%%%%%%%%%%%%%%%%%%%%%%%%%%%%%%%%%%%%%%%%%%

\subsection{Positive integral $\tmod$-modules}\label{subsection:taft-n-modules}

%%%%%%%%%%%%%%%%%%%%%%%%%%%%%%%%%%%%%%%%%%%%%%%%%%%

Let $\N=\Z_{\geq 0}$.
An $\N$-basis of an algebra is a basis
with structure constants from $\N$. Here is the
key example for our purposes which we
will always fix:

Denote by $\taftb(\tmod)$ the following $\N$-basis of the algebra $\taftp:=\chi(\tmod)_{\Z}\otimes_{\Z}\K$
\[
\{
[\tsimple],[\tinde],[\tprmod]
\mid 2\leq \ell\leq n-1, r\in \Z/n\Z
\}.
\]
The algebra $\taftp$ is therefore of dimension $n^{2}$.

\begin{definition}\label{definition:taft-pos-int-module}
A $\taftp$-module $\taftm$ is called an $\N$-module
if it has a finite $\K$-basis $\taftb(\taftm)$ on which elements
from $\taftb(\tmod)$ act by matrices with coefficients in $\N$. We additionally require that these integral matrices define a representation of $\taftp_{\Z}$. Two such $\N$-modules are $\N$-equivalent, denoted by $\cong_{\N}$, if there exists a permutation matrix which intertwines the $\taftb(\tmod)$-actions.
\end{definition}

\begin{definition}\label{definition:irreducible-module}
An $\N$-matrix (matrix with coefficients in $\N$) is called
\textit{irreducible} if it is the adjacency matrix of a
directed, strongly connected graph $\Gamma=(V(\Gamma), E(\Gamma))$
with finitely many vertices, for some enumeration of the
vertices.
\end{definition}

To fix notation, we also write $\Gamma^{\mathrm{op}}$ for the graph with the opposite orientation, and $\Gamma \coprod \Gamma^{\mathrm{op}}$ for the graph $(V(\Gamma), E(\Gamma) \coprod E(\Gamma^{\mathrm{op}}))$.

\begin{example}\label{example:taft-circle}
An important example is the circle graph $\Gamma=\Theta_{d}$ on $d$ vertices, where $d\in \N$ vertices
\[
\xy
(0,0)*{
\Theta_{1}=
\begin{tikzcd}[ampersand replacement=\&,row sep=tiny,column sep=tiny,arrows={shorten >=-.5ex,shorten <=-.5ex},labels={inner sep=.05ex}]
\phantom{.}
\\
\mathtt{1}\arrow[thick,out=120,in=60,loop,distance=1.25em,->]{}{}
\\
\phantom{.}
\end{tikzcd}};
(0,14)*{\text{\tiny $\amatrix[\Theta_{1}]=
\begin{psmallmatrix}
1
\end{psmallmatrix}$}};
(0,-10)*{\phantom{a}};
(0,17)*{\phantom{a}};
\endxy
,\quad
\xy
(0,0)*{
\Theta_{2}=
\begin{tikzcd}[ampersand replacement=\&,row sep=tiny,column sep=tiny,arrows={shorten >=-.5ex,shorten <=-.5ex},labels={inner sep=.05ex}]
\phantom{.} \& \phantom{.}
\\
\mathtt{1}\arrow[thick,yshift=-.1cm,->]{r}{} \& \mathtt{2}\arrow[thick,yshift=.1cm,->]{l}{}
\\
\phantom{.} \& \phantom{.}
\end{tikzcd}};
(0,14)*{\text{\tiny $\amatrix[\Theta_{2}]=
\begin{psmallmatrix}
0 & 1\\
1 & 0
\end{psmallmatrix}$}};
(0,-10)*{\phantom{a}};
(0,17)*{\phantom{a}};
\endxy
,\quad
\xy
(0,0)*{
\Theta_{3}=
\begin{tikzcd}[ampersand replacement=\&,row sep=tiny,column sep=tiny,arrows={shorten >=-.5ex,shorten <=-.5ex},labels={inner sep=.05ex}]
\phantom{.} \& \mathtt{2}\arrow[thick,->]{dd}{}
\\
\mathtt{1}\arrow[thick,->]{ru}{} \& \phantom{.}
\\
\phantom{.} \& \mathtt{3}\arrow[thick,->]{lu}{}
\end{tikzcd}};
(0,14)*{\text{\tiny $\amatrix[\Theta_{3}]=
\begin{psmallmatrix}
0 & 0 & 1\\
1 & 0 & 0\\
0 & 1 & 0
\end{psmallmatrix}$}};
(0,-10)*{\phantom{a}};
(0,17)*{\phantom{a}};
\endxy
,\quad
\xy
(0,0)*{
\Theta_{4}=
\begin{tikzcd}[ampersand replacement=\&,row sep=tiny,column sep=tiny,arrows={shorten >=-.5ex,shorten <=-.5ex},labels={inner sep=.05ex}]
\phantom{.} \& \mathtt{2}\arrow[thick,->]{dr}{} \& \phantom{.}
\\
\mathtt{1}\arrow[thick,->]{ru}{} \& \phantom{.} \& \mathtt{3}\arrow[thick,->]{ld}{}
\\
\phantom{.} \& \mathtt{4}\arrow[thick,->]{lu}{} \& \phantom{.}
\end{tikzcd}};
(0,14)*{\text{\tiny $\amatrix[\Theta_{4}]=
\begin{psmallmatrix}
0 & 0 & 0 & 1\\
1 & 0 & 0 & 0\\
0 & 1 & 0 & 0\\
0 & 0 & 1 & 0
\end{psmallmatrix}$}};
(0,-10)*{\phantom{a}};
(0,17)*{\phantom{a}};
\endxy
.
\]
We use the indicated enumeration and regard the adjacency matrix $\amatrix[\Theta_{d}]$ as acting on column vectors.
\end{example}

By \fullref{proposition:taft-tensor}, the action of
$\xpar=[\tinde[2,0]]$ and $\ypar=[\tsimple[1]]$ on any $\N$-module are given
by commuting $\N$-matrices
which we denote by $\mmatrix[\xpar]$
and $\mmatrix[\ypar]$, respectively. Similar notation will be used for all
other basis elements from $\taftb(\tmod)$.

\begin{definition}\label{definition:taft-transitive}
We call an $\N$-module:
\begin{itemize}
\item \textit{transitive} if $\xpar+\ypar$ acts by an irreducible $\N$-matrix.
\item \textit{Taft-split} if $\mmatrix[\xpar]=\mmatrix[\ypar]+\imatrix$.
\item \textit{NIM} (nonnegative integer matrix) if $\mmatrix[\zpar^{\ast}]=\mmatrix[\zpar]^{\mathrm{T}}$, for each $\zpar \in \taftb(\tmod)$.
\item \textit{admissible} if $\xpar$ and $\ypar$
act by diagonalizable matrices over $\C$.
\end{itemize}
\end{definition}

Note that Taft-split implies NIM, since $\mmatrix[\xpar^{\ast}]=\mmatrix[\ypar]^{-1}\mmatrix[\xpar]=\mmatrix[\ypar]^{-1}(\mmatrix[\ypar]+\imatrix)=\imatrix + \mmatrix[\ypar]^{\mathrm{T}}$.
We will see that being Taft-split has a categorical interpretation. Indeed, if $\catstuff{M}$ is an exact $\tmod$-module category, then the split Grothendieck group of $\catstuff{M}$ is Taft-split if and only if $\catstuff{M}$ is semisimple.

The use of nonnegative integer matrix representations and their associated graphs is already well explored (see, for example, \cite{Ga02} for the fusion-ring and fusion-graph point of view). In the context of tensor categories, the relationship between module categories, adjacency matrices and $\typeADET$ graphs plays a central role in the classification of module categories over $\mathrm{SL}_q(2)$; see \cite{EO04-2}. Closely related graph-theoretic methods occur in the theory of simple $2$-representations, notably in \cite{MT19} and \cite{MMMT20}, the latter of which also develops the language of nonnegative integral representations explicitly.

\begin{remark}
Note that transitivity, as defined above, implies that for
each $m,m^{\prime}\in\taftb(\taftm)$ there exists
an element $t\in\taftb(\tmod)$ such that $m$ appears in $t\acts m^{\prime}$
with nonzero coefficient when expressed in terms of $\taftb(\taftm)$.
\end{remark}

\begin{remark}\label{remark:perm-matrix}
The matrix $\mmatrix[\ypar]$ is a permutation matrix, since the coefficients of $\mmatrix[\ypar]$ are all non-negative integers and $\mmatrix[\ypar]^n=\imatrix$.
\end{remark}

\begin{definition}\label{definition:taft-regular}
We call an $\N$-module \textit{regular} if all disjoint cycles of the permutation associated to $\mmatrix[\ypar]$ have the same length.
\end{definition}

Recall from \cite{CM19} that any transitive $\N$-module has an associated apex. By definition, it is the highest cell that acts nontrivially. By \fullref{lemma:taft-cells} and \eqref{eq:taft-apex}
we get the following characterisation.

\begin{lemmaqed}\label{lemma:taft-apex-matrix}
Let $\taftm$ be a transitive $\N$-module of $\taftp$. It has
apex $\cell[b]$ if and only if $\cpoly[n](\mmatrix[\xpar],\mmatrix[\ypar])=0$.
Otherwise, the apex is $\cell[u]$.
\end{lemmaqed}

For any $d\in\Z_{> 0}$ with $n=dk$ we define a
$\taftp$-module $\taftm(d)=\K\{m_{1},\dots,m_{d}\}$ by
\[
\xy
(0,0)*{
\xpar\acts m_{i}
=
m_{i}+m_{i+1}
,\quad
\ypar\acts m_{i}
=
m_{i+1}};
\endxy
\]
where the subscripts are taken modulo $d$.

\begin{example}\label{example:taft-big-apex}
If $n=6$, then for $d\in\{1,2,3,6\}$, the action of $\xpar$ and $\ypar$ on $\taftm(d)$ is respectively given by $\imatrix+\amatrix[\Theta_{d}]$
(cf. \fullref{example:taft-circle}) and $\amatrix[\Theta_{d}]$.
\end{example}

\begin{proposition}\label{proposition:taft-big-apex}
For each admissible, transitive, Taft-split $\N$-module $\taftm$ with apex $\cell[u]$, there exists a unique $d|n$ such that $\taftm \cong_{\N} \taftm(d)$.
\end{proposition}
\begin{proof}
First, for $d|n$, let us show that $\taftm(d)$ has the claimed properties. By \fullref{proposition:taft-ring}, it is enough to prove $\mmatrix[\xpar]\mmatrix[\ypar]=\mmatrix[\ypar]\mmatrix[\xpar]$,
$\mmatrix[\ypar]^{n}=\imatrix$ and $\mmatrix[\xpar]-\mmatrix[\ypar]-\imatrix=0$, which follows
directly from the definition of $\taftm(d)$. Moreover, all involved matrices are
clearly diagonalizable, and transitivity is also clear.
That they correspond to the apex $\cell[u]$ follows from \fullref{lemma:taft-apex-matrix} since $\cpoly[n](\mmatrix[\ypar]+\imatrix)=\sum_{i=0}^{n-1}{\mmatrix[\ypar]^i}\neq 0$. The fact that it is an $\N$-module follows from \fullref{proposition:taft-decat}.

\medskip
Second, assume that $\taftm$ is an admissible, transitive, Taft-split $\N$-module with apex $\cell[u]$.
We know that $\mmatrix[\ypar]$ is a permutation matrix (cf. \fullref{remark:perm-matrix}).
 The relation $\mmatrix[\xpar]=\mmatrix[\ypar]+\imatrix$ implies that the matrix $\mmatrix[\xpar]+\mmatrix[\ypar] = 2\mmatrix[\ypar] + \imatrix$ is irreducible exactly when the permutation associated to $\mmatrix[\ypar]$ consists of a single cycle. If $d$ is the length of this cycle, we must have $d|n$ since $\mmatrix[\ypar]^n=\imatrix$, and the module is $\N$-equivalent to $\taftm(d)$.

\medskip
Third, for different $d,d^{\prime}$, the
$\taftp$-modules $\taftm(d)$ and $\taftm(d^{\prime})$ are of different dimensions,
and therefore cannot be $\N$-equivalent, hence uniqueness.
\end{proof}

\begin{example}
For $n=2$, consider the $\taftp$-module defined by:
\begin{equation*}
\mmatrix[\xpar]=\begin{pmatrix} 1 & 1\\ 1& 1\end{pmatrix}, \quad \mmatrix[\ypar]=\imatrix_2.
\end{equation*}
Indeed, it is an $\N$-module since $(\mmatrix[\xpar]-\mmatrix[\ypar]-\imatrix_2)\cpoly[2](\mmatrix[\xpar],\mmatrix[\ypar])=(\mmatrix[\xpar]-2\imatrix_2)\mmatrix[\xpar]=0$ and the matrices $\mmatrix[\ypar]$ and $\mmatrix[\ypar]$ commute. It is moreover clearly admissible and transitive. The fact that it is NIM because
\begin{equation*}
\mmatrix[\xpar^{\ast}]=\mmatrix[\ypar]^{-1}\mmatrix[\xpar]=\mmatrix[\xpar], \quad \mmatrix[\ypar^{\ast}]=\mmatrix[\ypar]^{\smin 1}=\mmatrix[\ypar].
\end{equation*}
Most importantly, this $\taftp$-module is not Taft-split. Moreover, it is clearly not $\N$-equivalent to $\taftm(1)$ since it is two-dimensional and is also not $\N$-equivalent to $\taftm(2)$, since the cycle type underlying the permutation of $\imatrix$ is $(1)(2)$, but the cycle type underlying $\mmatrix[\ypar]$ of $\taftm(2)$ is $(1\text{ }2)$.
\end{example}

Let us now turn to $\N$-modules with apex $\cell[b]$. To that end, let us consider the families of oriented graphs of type $\typeADET$ oriented in the following way:
\medskip
\[
A_{\ell}:\qquad
\begin{tikzcd}[column sep=1.2em]
1 \arrow[thick,->]{r}{}
& 2 \arrow[thick,->]{r}{}
& \cdots \arrow[thick,->]{r}{}
& \ell\smin 1 \arrow[thick,->]{r}{}
& \ell
\end{tikzcd}
\]

\[
D_{\ell}:\qquad
\begin{tikzcd}[column sep=1.2em, row sep=0.8em]
&&&& \ell\smin 1 \\
1  \arrow[thick,->]{r}{}
& 2  \arrow[thick,->]{r}{}
& \cdots \arrow[thick,->]{r}{}
& \ell \smin 2
\arrow[thick,->]{ur}{}
\arrow[thick,->]{dr}{}\\
&&&& \ell
\end{tikzcd}
\]

\[
T_{\ell}:\qquad
\begin{tikzcd}[column sep=1.2em]
1 \arrow[thick,->]{r}{}
& 2 \arrow[thick,->]{r}{}
& \cdots \arrow[thick,->]{r}{}
& \ell\smin 1 \arrow[thick,->]{r}{}
& \ell \arrow[thick, ->,loop right]{}{}
\end{tikzcd}
\]

\[
E_6:\qquad
\begin{tikzcd}[column sep=1.2em, row sep=0.8em]
&& 4  \\
1  \arrow[thick,->]{r}{}
& 2  \arrow[thick,->]{r}{}
& 3  \arrow[thick,->]{r}{} \arrow[thick,->]{u}{}
& 5  \arrow[thick,->]{r}{}
& 6
\end{tikzcd}
\]

\[
E_7:\qquad
\begin{tikzcd}[column sep=1.2em, row sep=0.8em]
&&& 5\\
1  \arrow[thick,->]{r}{}
& 2  \arrow[thick,->]{r}{}
& 3  \arrow[thick,->]{r}{}
& 4  \arrow[thick,->]{r}{} \arrow[thick,->]{u}{}
& 6  \arrow[thick,->]{r}{}
& 7
\end{tikzcd}
\]

\[
E_8:\qquad
\begin{tikzcd}[column sep=1.2em, row sep=0.8em]
&&&& 6\\
1  \arrow[thick,->]{r}{}
& 2  \arrow[thick,->]{r}{}
& 3  \arrow[thick,->]{r}{}
& 4  \arrow[thick,->]{r}{}
& 5  \arrow[thick,->]{r}{} \arrow[thick,->]{u}{}
& 7  \arrow[thick,->]{r}{}
& 8
\end{tikzcd}
\]
\medskip

We recall the Coxeter number $h(\Gamma)$ of each $\typeADET$ diagram.

\medskip
\begin{center}
\begin{tabular}{|c|c|}
\hline
\text{Type} $\Gamma$ & $h(\Gamma)$\\
\hline
$\typeA_{\ell}$& $\ell+1$ \\
\hline
$\typeD_{\ell}$  & $2(\ell-1)$ \\
\hline
$\typeT_{\ell}$  & $2\ell+1$ \\
\hline
$\typeE_6$ & $12$\\
\hline
$\typeE_7$ & $18$\\
\hline
$\typeE_8$ & $30$\\
\hline
\end{tabular}
\end{center}
\medskip

Let us denote by $\typeADET(n)$ the set of pairs $(\Gamma,d)$ such that $\Gamma$ is an $\typeADET$ graph with $h(\Gamma)=n$ and $d \in \N$ is a divisor of $n$.

%Fix a finite oriented graph $\Gamma$, possibly with loops whose symmetrization $\Gamma \coprod \Gamma^{\mathrm{op}}$ is strongly connected. Let us enumerate the vertices $1,\dots,v(\Gamma)$ and fix $d\in\Z_{> 0}$ with $n=dk$, such that either $\Gamma$ has no loops or $d$ is odd and $\Gamma$ has at most one loop per vertex.

Associated to $(\Gamma,d) \in \typeADET(n)$ we define the $\taftp$-module $\taftm(\Gamma,d)$
of dimension $v(\Gamma)d$ as follows. It has a $\K$-basis
$\{m_{ij}\mid (i,j) \in  \llbracket 1, v(\Gamma)\rrbracket \times \Z/d\Z\}$ and $\taftp$-action
\[
\xy
(0,0)*{
\xpar\acts m_{ij}
=
{\textstyle\sum_{\substack{(i\to h)\in\Gamma\\ \hspace*{-0.4cm}i \neq h}} m_{hj}}
+
{\textstyle\sum_{\substack{(i\to h)\in\Gamma^{\mathrm{op}}\\ \hspace*{-0.7cm} i \neq h}} m_{hj{+}1}}
+
{\textstyle\sum_{\substack{(i\to i)\in\Gamma}} \epsilon_im_{ij{+}\frac{d+1}{2}}}
,\quad
\ypar\acts m_{ij}
=
m_{ij{+}1}}
\endxy
\]
where $\epsilon_i$ is equal to $1$ if $\Gamma$ has a loop at vertex $i$ and $0$ otherwise.

\begin{remark}
If $\tilde{\Gamma}$ is obtained from $\Gamma$ by changing the orientation, note that ${\taftm(\tilde{\Gamma},d) \cong_{\N} \taftm(\Gamma,d)}$. In particular, the choice of orientations of the $\typeADET$ graphs plays no role.
\end{remark}

\begin{example}\label{example:taft-small-apex}
To understand the $\taftp$-module $\taftm(\Gamma,d)$, let us consider an example for $n=6$, where
we take the type $\typeD_{4}$ diagram.
\[
\xy
(0,0)*{
\Gamma=
\begin{tikzcd}[ampersand replacement=\&,row sep=tiny,column sep=tiny,arrows={shorten >=-.5ex,shorten <=-.5ex},labels={inner sep=.05ex}]
\phantom{.} \& \phantom{.} \& \mathtt{3}
\\
\mathtt{1}\arrow[thick,->]{r}{} \& \mathtt{2}\arrow[thick,->]{rd}{}\arrow[thick,->]{ru}{} \& \phantom{1}
\\
\phantom{.} \& \phantom{.} \& \mathtt{4}
\end{tikzcd}\quad\quad\&\quad\quad
\Gamma^{\mathrm{op}}=
\begin{tikzcd}[ampersand replacement=\&,row sep=tiny,column sep=tiny,arrows={shorten >=-.5ex,shorten <=-.5ex},labels={inner sep=.05ex}]
\phantom{.} \& \phantom{.} \& \mathtt{3}
\\
\mathtt{1}\arrow[thick,<-]{r}{} \& \mathtt{2}\arrow[thick,<-]{rd}{}\arrow[thick,<-]{ru}{} \& \phantom{1}
\\
\phantom{.} \& \phantom{.} \& \mathtt{4}
\end{tikzcd}};
(0,14)*{\text{\tiny $\amatrix[\Gamma]=
\begin{psmallmatrix}
0 & 0 & 0 & 0\\
1 & 0 & 0 & 0\\
0 & 1 & 0 & 0\\
0 & 1 & 0 & 0
\end{psmallmatrix}\quad\quad\&\quad\quad
\amatrix[\Gamma^{\mathrm{op}}]=
\begin{psmallmatrix}
0 & 1 & 0 & 0\\
0 & 0 & 1 & 1\\
0 & 0 & 0 & 0\\
0 & 0 & 0 & 0
\end{psmallmatrix}$}};
(0,-14)*{\text{\tiny $\amatrix[\Gamma \mathbin{\scalebox{0.7}{$\coprod$}} \Gamma^{\mathrm{op}}]=
\begin{psmallmatrix}
0 & 1 & 0 & 0\\
1 & 0 & 1 & 1\\
0 & 1 & 0 & 0\\
0 & 1 & 0 & 0
\end{psmallmatrix}$}};
\endxy
\]
The adjacency matrices are as above. The number of vertices is
$v(\Gamma)=4$, and the Coxeter number is $h(\Gamma)=6$. Let us
choose $d=3$. Then we get the following $12\times 12$ matrices
\[
\mmatrix[\xpar]=
\scalebox{.8}{$\displaystyle
\left(
\begin{array}{c|c|c|c}
\zmatrix & \amatrix[\Theta_{3}] & \zmatrix & \zmatrix
\\
\hline
\imatrix & \zmatrix & \amatrix[\Theta_{3}] & \amatrix[\Theta_{3}]
\\
\hline
\zmatrix & \imatrix & \zmatrix & \zmatrix
\\
\hline
\zmatrix & \imatrix & \zmatrix & \zmatrix
\end{array}
\right)
$}
,\quad
\mmatrix[\ypar]=
\scalebox{.8}{$\displaystyle
\left(
\begin{array}{c|c|c|c}
\amatrix[\Theta_{3}] & \zmatrix & \zmatrix & \zmatrix
\\
\hline
\zmatrix & \amatrix[\Theta_{3}] & \zmatrix & \zmatrix
\\
\hline
\zmatrix & \zmatrix & \amatrix[\Theta_{3}] & \zmatrix
\\
\hline
\zmatrix & \zmatrix & \zmatrix & \amatrix[\Theta_{3}]
\end{array}
\right)
$}
\]
as the action matrices for $\xpar$ and $\ypar$. The matrix $\mmatrix[\xpar]$ is
obtained from $\amatrix[\Gamma\coprod\Gamma^{\mathrm{op}}]$ by replacing
all nonzero entries in $\amatrix[\Gamma]$ by $3 \times 3$ identity matrices,
and all such entries of $\amatrix[\Gamma^{\mathrm{op}}]$
by the adjacency matrix $\amatrix[\Theta_{d}]$
(cf. \fullref{example:taft-circle}). One can then check that
these give a $\taftp$-module, such that $\cpoly[6](\mmatrix[\xpar],\mmatrix[\ypar])=0$ holds.
\end{example}

\begin{example}
\label{ex:non-reg}
If $n=4$, consider the $\N$-module defined by
\[
\mmatrix[\xpar]=
\scalebox{.8}{$\displaystyle
\left(
\begin{array}{ccc}
0 & 1 & 1
\\
1 & 0 & 0
\\
1 & 0 & 0
\end{array}
\right)
$}
,\quad
\mmatrix[\ypar]=
\scalebox{.8}{$\displaystyle
\left(
\begin{array}{ccc}
1 & 0 & 0
\\
0 & 0 & 1
\\
0 & 1 & 0
\end{array}
\right)
$}.
\]
It is admissible, transitive, NIM, with apex $\cell[b]$, but not regular.
\end{example}

\begin{proposition}\label{proposition:taft-small-apex}
For each admissible, regular, transitive, NIM $\N$-module $\taftm$ with apex $\cell[b]$, there exists a unique pair $(\Gamma,d) \in \typeADET(n)$ such that $\taftm \cong_{\N} \taftm(\Gamma,d)$.
\end{proposition}

\begin{proof}
First, for $(\Gamma,d) \in \typeADET(n)$, let us show that $\taftm(\Gamma,d)$ has the claimed properties.
 Observe that $\mmatrix[\xpar]\mmatrix[\ypar]=\mmatrix[\ypar]\mmatrix[\xpar]$,
which follows since $\mmatrix[\ypar]$ is a block-diagonal matrix whose
block entries commute with all block entries of $\mmatrix[\xpar]$. Moreover $\mmatrix[\ypar]^{n}=\imatrix$ holds by construction and $\mmatrix[\ypar]$
is evidently diagonalizable. That $\mmatrix[\xpar]$ is diagonalizable over $\C$
follows from
\[
\left(
\mmatrix[\xpar]\mmatrix[\ypar]^{-1}=\mmatrix[\xpar]^{\mathrm{T}}=\mmatrix[\ypar]^{-1}\mmatrix[\xpar]
\quad\&\quad
\mmatrix[\xpar]\mmatrix[\ypar]=\mmatrix[\ypar]\mmatrix[\xpar]
\right)
\Rightarrow
\left(
\mmatrix[\xpar]\mmatrix[\xpar]^{\mathrm{T}}=\mmatrix[\xpar]^{\mathrm{T}}\mmatrix[\xpar]
\right),
\]
proving admissibility. Regularity follows directly by construction since every disjoint cycle has length $d$.
Transitivity follows by construction since the underlying $\typeADET$ diagram is connected, and
the two directed actions associated with each edge connect the corresponding $\ypar$-orbits.
The fact that it is an $\N$-module follows from \fullref{prop:decat-coherence}.
By \fullref{lemma:taft-apex-matrix}, it remains to show that $\cpoly[n](\mmatrix[\xpar],\mmatrix[\ypar])=0$ holds.
To this end, recall that $n=dk$ and that
the spectra of the type $\typeADET$ graphs are
known and contained in $(\smin2,2)$.
%and multiplicity free up to the
%type $\typeD$ graphs which have a double eigenvalue zero in
%the case of an even number of vertices (which happens
%only if $\cpoly[n](\xpar,1)$ has no constant term), cf. \cite{Smith70}.
For each $d^{\mathrm{th}}$ root of unity $\omega$, let $V_{\omega}$ denote the $\omega$-eigenspace of $\amatrix[\Theta_{d}]$. Denote also by $\mstuff{D}:= \delta^{\typeT}_{\Gamma}\mathrm{diag}(0, 0, \dots,1)$ where $\delta^{\typeT}_{\Gamma}$ is equal to $1$ if $\Gamma$ is of type $\typeT$ and $0$ otherwise. The restriction of $\mmatrix[\xpar]$ to $\K^{v(\Gamma)}\otimes V_{\omega}$ is
\begin{equation}
{\mmatrix[\xpar]}_{|\K^{v(\Gamma)}\otimes V_{\omega}} = \amatrix[\Gamma]-\mstuff{D} + \omega (\amatrix[\Gamma]^{\mathrm{T}}-\mstuff{D}) +  \omega^{\frac{d+1}{2}}\mstuff{D}.
\end{equation}
Let us denote by $C_{\Gamma}:=\amatrix[\Gamma] + \amatrix[\Gamma]^{\mathrm{T}}- \mstuff{D}$ the adjacency matrix of the diagram $\Gamma$.
Having this, we observe that the joint eigenvalues of $\mmatrix[\xpar]$ and $\mmatrix[\ypar]$ in type $\typeADET$ is
\begin{equation*}\label{eq:taft-spectra}
\left\{(\sigma_{\omega}\tau, \omega) \in \C^2 \mid \omega^d=1, \tau \in \mathrm{Spec}(C_{\Gamma}) \right\}
\end{equation*}
where $\sigma_{\omega}$ is a square root of $\omega$.
In the end, we have $U_n(\sigma_{\omega}\tau,\omega)=\sigma_{\omega}^{n\smin 1}U_n(\tau,1)=0$.% since $C_{\Gamma}$ is the adjacency matrix of $\Gamma$, which is of type $\typeADET$ with $h(\Gamma)=n$.

\medskip
Second let $\taftm$ be an admissible, regular, transitive, NIM $\N$-module such that ${\cpoly[n](\mmatrix[\xpar],\mmatrix[\ypar])=0}$. Since $\mmatrix[\ypar]$ is the permutation matrix of a regular permutation, we can choose a basis of $\taftm$ such that $\mmatrix[\ypar]$ is the following Kronecker product $\imatrix \otimes \amatrix[\Theta_{d}]$, for $d$ the length of a disjoint cycle. This basis gives us an identification as $\K$-vector spaces of $\taftm$ with $\K^p \otimes_{\K} \K^d$ where $p$ is the number of cycles. Moreover, $\mmatrix[\xpar]$ and $\mmatrix[\ypar]$  commute. This implies that there exists $d$ matrices $\batrix[j] \in M_{p}(\K)$ such that $\mmatrix[\xpar]$ is the sum of Kronecker products:
\[
\mmatrix[\xpar]=\sum_{j=0}^{d-1}{\batrix[j] \otimes \amatrix[\Theta_{d}]^j}.
\]
Let $\batrix:=\sum_{j=0}^{d-1}{\batrix[j]}$. We want to argue that the spectral radius of $\batrix$ is strictly smaller than 2, so that using transitivity, we can apply \cite[Theorem 10]{McS07} (see also \cite[Theorem 2]{Smith70} or \cite[Theorem 1.4.3]{GHJ89} applied to the bipartite double) and obtain that $\batrix$ is the adjacency matrix of an $\typeADET$ graph. Indeed, transitivity implies that $\batrix$ is indecomposable and the fact that the $\N$-module is NIM, gives that $\batrix$ is symmetric.

To do that, we know that $\mmatrix[\xpar]$ and $\mmatrix[\ypar]$ are both diagonalizable over $\C$ and are commuting. Consider a basis where these two matrices are diagonal. Let us denote by $\lb_i$ the eigenvalues of $\mmatrix[\xpar]$, and by $\mu_i$ those of $\mmatrix[\ypar]$. Since $\mmatrix[\ypar]^n=\imatrix$, every $\mu_i$ is an $n^{\mathrm{th}}$ root  of unity. The condition $\cpoly[n](\mmatrix[\xpar],\mmatrix[\ypar])=0$ is equivalent to $\cpoly[n](\lb_i,\mu_i)=0, \forall i \in \llbracket 1, \mathrm{dim}(\taftm) \rrbracket$. Now, \cite[Lemma 3.11]{COZ14} gives:

\begin{align*}
\cpoly[n](\lb_i,\mu_i) &= \sum_{j=0}^{\lfloor \frac{n-1}{2} \rfloor}{(-1)^j \binom{n-1-j}{j}{\lb_i}^{n-1-2j}{\mu_i}^{j}}\\
                 &= {\mu_i}^{\frac{n \smin 1}{2}} \sum_{j=0}^{\lfloor \frac{n-1}{2} \rfloor}{(-1)^j \binom{n-1-j}{j}{\left(\frac{\lb_i}{\sqrt{\mu_i}}\right)}^{n-1-2j}}\\
                 &= {\mu_i}^{\frac{n \smin 1}{2}} \ctpoly[n-1]\left(\frac{\lb_i}{2\sqrt{\mu_i}}\right)
\end{align*}
where $\ctpoly[n]$ is the Chebyshev polynomial of the second kind and can be defined as follows:
\[
\ctpoly[n](x) := \sum_{j=0}^{\lfloor \frac{n}{2} \rfloor}{(-1)^j \binom{n-j}{j}{(2x)}^{n-2j}}.
\]
The roots of $\ctpoly[n]$ are $\left(\mathrm{cos}\left(\frac{j\pi}{n+1}\right)\right)_{j \in \llbracket 1,n \rrbracket}$. For every $i$, there exists $j \in \llbracket 1,n \smin 1 \rrbracket$ such that:
\[
  \lambda_i=2\sqrt{\mu_i}\mathrm{cos}\left(\frac{j\pi}{n}\right).
\]
Hence $|\lambda_i| < 2$. Now let $V:= \mathrm{Ker}(\mmatrix[\ypar]- \imatrix)$. Since $\mmatrix[\xpar]$ and $\mmatrix[\ypar]$ commute, $V$ is $\mmatrix[\xpar]$ stable. By construction, the restriction of $\mmatrix[\xpar]$ to $V$ is precisely $\batrix$. This implies that every eigenvalue of $\batrix$ is an eigenvalue of $\mmatrix[\xpar]$. We now have that $\batrix$ is the adjacency matrix of an $\typeADET$ graph $\Gamma$. There remains now to show that $h:=h(\Gamma)$ is equal to $n$ and that the action of $\taftp$ on $\taftm$ is exactly the one defined for $\taftm(\Gamma,d)$.

Restricting the equality $\cpoly[n](\mmatrix[\xpar],\mmatrix[\ypar])=0$ to $V$, we get $\cpoly[n](\batrix,\imatrix)=0$. Thanks to \cite[Theorem 1.4.3]{GHJ89}, we know that if $v$ denotes the Perron-Frobenius eigenvector of $\batrix$, then:
\begin{equation*}
\batrix v = 2\mathrm{cos}\left(\frac{\pi}{h}\right)v.
\end{equation*}
This result still holds for type $\typeT_{\ell}$ since it can be constructed as a folding of the graph of type $\typeA_{2\ell}$. Now, the fact that $\cpoly[n](\batrix,\imatrix)v=\cpoly[n]\left(2\mathrm{cos}\left(\frac{\pi}{h}\right),1\right)v=0$ implies that $\mathrm{sin}\left(\frac{n\pi}{h}\right)=0$. Indeed, by definition we have
\begin{equation*}
\cpoly[n](2\mathrm{cos}\left(\frac{\pi}{h}\right),1)=\frac{\mathrm{sin}\left(n\pi/h\right)}{\mathrm{sin}\left(\pi/h\right)}.
\end{equation*}
As a result we get that $h |n$. Let us now suppose, for contradiction, that $h<n$. Because $h \geq 2$, we have that $2 \leq h+1 \leq n-1$. Let $\mmatrix[h]:=\cpoly[h+1](\mmatrix[\xpar],\mmatrix[\ypar])$. Using \fullref{proposition:taft-ring}, it is the action matrix of $\tinde[h+1,0]$ and hence must be an $\N$-matrix. We get a contradiction because $\mmatrix[h]v=\cpoly[h+1](\batrix,\imatrix)v=\frac{\mathrm{sin}\left((h+1)\pi/h\right)}{\mathrm{sin}\left(\pi/h\right)}v=-v$.

The fact that $\taftm$ has the desired $\taftp$-action follows from the fact that the module is NIM. Indeed, we get that $\mmatrix[\xpar]^{\mathrm{T}}=\mmatrix[\ypar]^{-1}\mmatrix[\xpar]$, hence that
\begin{equation*}
\batrix[j]^{\mathrm{T}}=\batrix[1-j],\qquad \forall j \in \Z/d\Z.
\end{equation*}
For type $\typeADE$, the underlying graph being a tree, there are no cycles and the reversed edge has then a shift of $1$. Finally, for the loop of type $\typeT$, if the cyclic shift is denoted by $s$, then $\batrix[j]^{\mathrm{T}}=\batrix[1-j]$ implies that $s \equiv 1-s \mod d$, so $s=\frac{d+1}{2}$.
\medskip

Third, uniqueness is a direct consequence of these two facts: the cycle decomposition of $\mmatrix[\ypar]$ determines $d$ and the set of $\ypar$-orbits. Moreover, the induced action of $\xpar$ on the orbit sums determines the underlying diagram.
\end{proof}

\section{Module categories of $\tmod$}\label{section:taft-rep}

We retain the conventions from \fullref{section:taft}. Thus $n\in\Z_{\geq 2}$, $\K$ is an algebraically closed field whose characteristic does not divide $n$ and $\qpar\in\K$ is a primitive $n^{\mathrm{th}}$ root of unity.

%%%%%%%%%%%%%%%%%%%%%%%%%%%%%%%%%%%%%%%%%%%%%%%%%%%

\subsection{Algebra objects}\label{subsection:taft-algebra-objects}

%%%%%%%%%%%%%%%%%%%%%%%%%%%%%%%%%%%%%%%%%%%%%%%%%%%
We recall the definition from \cite[Section 7.8]{EGNO} of an algebra object.
\begin{definition}
An algebra object $\taobj[]$
in a monoidal category $\catstuff{C}$ is an object of $\catstuff{C}$
together with a unit and multiplication morphisms $\morstuff{u}\colon 1\to\taobj[]$ and $\morstuff{m}\colon\taobj[]\otimes\taobj[]\to\taobj[]$ satisfying the axioms:
\begin{equation*}
\begin{cases}
\morstuff{m} \circ (\morstuff{m} \otimes \mathrm{id}_{\taobj[]}) = \morstuff{m} \circ (\mathrm{id}_{\taobj[]} \otimes \morstuff{m}),\\
\morstuff{m} \circ (\morstuff{u} \otimes \mathrm{id}_{\taobj[]}) = \mathrm{id}_{\taobj[]} = \morstuff{m} \circ (\mathrm{id}_{\taobj[]} \otimes \morstuff{u}).
\end{cases}
\end{equation*}

For such an algebra object the notion of a (right) module object $\xmodule$ in $\catstuff{C}$ is defined in the following way: an object of $\catstuff{C}$ together with a multiplication morphism
$\acts_{\taobj[]}\colon\xmodule\otimes\taobj[]\to\xmodule$, satisfying the axioms:
\begin{equation*}
\begin{cases}
\acts_{\taobj[]} \circ (\acts_{\taobj[]} \otimes \mathrm{id}_{\taobj[]}) = \acts_{\taobj[]} \circ (\mathrm{id}_{\xmodule} \otimes \morstuff{m}),\\
\acts_{\taobj[]} \circ (\mathrm{id}_{\xmodule} \otimes \morstuff{u}) = \mathrm{id}_{\xmodule}.
\end{cases}
\end{equation*}
The collection of such module objects forms a category which we denote by $\oneMod_{\catstuff{C}}(\taobj[])$.
\end{definition}

\begin{example}\label{example:taft-alg-objects-1}
If  $\catstuff{C}$ is a tensor category and $\tmodule$ an object of $\catstuff{C}$, the object $\taobj[\tmodule]=\tmodule\otimes\tmodule^{\ast}$
is an algebra object in $\catstuff{C}$ with unit morphism
$\mathrm{coev}_{\tmodule}\colon 1\to\tmodule\otimes\tmodule^{\ast}$ and multiplication morphism ${\mathrm{id}_{\tmodule} \otimes \mathrm{ev}_{\tmodule} \otimes \mathrm{id}_{\tmodule^{\ast}}\colon\tmodule\otimes\tmodule^{\ast}\otimes\tmodule\otimes\tmodule^{\ast}\to\tmodule\otimes\tmodule^{\ast}}$.
\begin{comment}
can be diagrammatically
presented by
\[
\morstuff{u}\colon 1\to\tmodule\otimes\tmodule^{\ast}=
\begin{tikzpicture}[anchorbase, scale=.4, tinynodes]
	\draw[dstrand,directed=.55] (2,2) node [above] {$\tmodule^{\ast}$} to [out=270, in=0] (1,1) to [out=180, in=270] (0,2) node [above] {$\tmodule$};
\end{tikzpicture}
,\quad\quad
\morstuff{m}\colon\tmodule\otimes\tmodule^{\ast}\otimes\tmodule\otimes\tmodule^{\ast}\to\tmodule\otimes\tmodule^{\ast}=
\begin{tikzpicture}[anchorbase, scale=.4, tinynodes]
	\draw[dstrand,directed=.55] (2,-2) node [below] {$\tmodule$} to [out=90, in=0] (1,-1) to [out=180, in=90] (0,-2) node [below] {$\tmodule^{\ast}$};
	\draw[dstrand,directed=.55] (-2,-2) node [below] {$\tmodule$} to [out=90, in=270] (0,0) node [above] {$\tmodule$};
	\draw[dstrand,directed=.55] (2,0) node [above] {$\tmodule^{\ast}$} to [out=270, in=90] (4,-2) node [below] {$\tmodule^{\ast}$};
\end{tikzpicture}
\]
where we use the string calculus (reading
left to right for $\otimes$ and bottom to top for $\circ$) for monoidal categories
to represent the morphisms. The algebra axioms then follow from the standard isotopies of the corresponding string diagrams.
\end{comment}
Moreover, if $\xmodule$ is a module object, then, for any object
$\tmodule[M^{\prime}]$ of $\catstuff{C}$, the object $\tmodule[M^{\prime}]\otimes\xmodule$ is
also a module object. In particular, $\tmodule\otimes\taobj[]$ is always a module object.
\end{example}

Let $\taobj[]$ be an algebra object in a finite tensor
category $\catstuff{C}$. There
is the forgetful functor $\morstuff{U}\colon\oneMod_{\catstuff{C}}(\taobj[])\to\catstuff{C}$
sending $\xmodule$ to its underlying object $\tmodule$ in $\catstuff{C}$, and the
free functor ${\morstuff{F}\colon\catstuff{C}\to\oneMod_{\catstuff{C}}(\taobj[])}$ is given by
$\tmodule\mapsto\tmodule\otimes\taobj[]$. By \cite[Lemma 7.8.12]{EGNO}, $\morstuff{U}$ is right
adjoint to $\morstuff{F}$, i.e. for any $\tmodule\in\catstuff{C}$
and $\xmodule[M]^{\prime}\in\oneMod_{\catstuff{C}}(\taobj[])$ there is a natural isomorphism
\[
\Hom_{\taobj[]}(\tmodule\otimes\taobj[],\xmodule[M]^{\prime})\cong
\Hom_{\catstuff{C}}(\tmodule,\morstuff{U}(\xmodule[M]^{\prime}))
=\Hom_{\catstuff{C}}(\tmodule,\tmodule[M]^{\prime}),
\]
where we write $\Hom_{\taobj[]}:=\Hom_{\oneMod_{\catstuff{C}}(\taobj[])}$ for short.

\begin{definition}\label{definition:simple-aobject}
An algebra object $\taobj[]$ in a category $\catstuff{C}$ is called simple if its only two-sided $\catstuff{C}$-ideals are $0$ and $\taobj[]$.
\end{definition}

%%%%%%%%%%%%%%%%%%%%%%%%%%%%%%%%%%%%%%%%%%%%%%%%%%%

\subsection{Simple algebra objects in $\tmod$ with upper cell apex}\label{subsection:taft-simple-algebra-objects}

%%%%%%%%%%%%%%%%%%%%%%%%%%%%%%%%%%%%%%%%%%%%%%%%%%%

Let $G=\langle\gat\mid\gat^{n}=1\rangle\cong\Z/n\Z$.

\begin{definition}\label{definition:taft-algebra-objects}\cite[Section 4.3]{EO04}
Let $\lambda\in\K$ and $d,k\in\Z_{\geq 0}$, with $n=dk$.
The algebra object $\taobj\in\tmod$ is defined as follows.
\begin{itemize}

\setlength\itemsep{.15cm}

\item As a vector space:
\[
\taobj=\K[G/H]\otimes_{\K}\K[\zapar]/(\zapar^{n}-\lambda\cdot 1),
\]
where $H\subset G$ is the subgroup of order $d$
generated by $\gat^{k}$.

\item The (left) $\taft$-action on $\taobj$ is given by
\begin{gather}\label{eq:taft-alg-ob-action}
\begin{gathered}
\gt\acts 1=1,\;\gt\acts\gat=\qpar^{-d}\gat,\;\gt\acts\zapar=\qpar^{-1}\zapar
,
\quad
\xx\acts 1=0,\;\xx\acts\gat=0,\;\xx\acts\zapar=1,
\\
\gt\acts(\obstuff{a}\obstuff{b})=(\gt\acts\obstuff{a})(\gt\acts\obstuff{b})
,
\quad
\xx\acts(\obstuff{a}\obstuff{b})
=(\xx\acts\obstuff{a})\obstuff{b}+(\gt\acts\obstuff{a})(\xx\acts\obstuff{b}),
\quad
\text{for all }\obstuff{a},\obstuff{b}\in\taobj.
\end{gathered}
\end{gather}

\item The algebra structure is defined by the products in $\K[G/H]$ and
$\K[\zapar]/(\zapar^{n}-\lambda\cdot 1)$,
with the relation $\gat\zapar=\qpar^{\smin d}\zapar\gat$.\qedhere

\end{itemize}
\end{definition}

\begin{example}\label{example:taft-algebra-objects}
As an object of $\tmod$, we have $\taobj[1{,}\lambda]\cong\taft$.
However, the multiplication is
deformed by the element $\lambda\in\K$. In particular,
$\taobj[1{,}0]\cong\taft$ (where the isomorphism is given in \fullref{remark:taft-def}) is non-semisimple, while
$\taobj[1{,}1]$ is isomorphic to the twisted tensor product $\K[\Z/n\Z]\otimes_{\qpar}\K[\Z/n\Z]$  which is semisimple cf. \cite[Definition 2.1]{CSV95}. The twisting map is given by $\tau(g\otimes z)=\qpar z\otimes g$.
\end{example}

 Using \eqref{eq:taft-alg-ob-action}, one obtains
\begin{gather}\label{eq:x-action}
\xx\acts(\gat^{i}\zapar^{j})=
\qpar^{-di}(1+\qpar^{{-}1}+\dots+\qpar^{1{-}j})\gat^{i}\zapar^{j{-}1}.
\end{gather}
In particular, $\xx\acts\zapar^{n}=0$ since $1+\qpar^{-1}+\dots+\qpar^{{-}n{+}1}=0$. The following result follows directly from the definition of $\taobj$.

\begin{lemma}\label{lemma:taft-simple-object}
The sets
\[
\{\gat^{i}\zapar^{j}\mid 0\leq i\leq k-1,\; 0\leq j\leq n-1\}
,\quad
\{\zapar^{j}\gat^{i}\mid 0\leq i\leq k-1,\; 0\leq j\leq n-1\}
\]
are two distinct bases of $\taobj$ when $k> 1$. In particular $\mathrm{dim}_{\K}(\taobj)=kn=n^2/d$.
\end{lemma}

\begin{comment}
\begin{proof}
The claim that $\taobj$ is an algebra object in $\tmod$ is an easy calculation and omitted.
Clearly, $\taobj$ is as a vector space of dimension $kn$ and the
elements $\gat^{i}\zapar^{j}$ are linearly independent.
\end{proof}
\end{comment}

\begin{proposition}\label{proposition:taft-simple-object}
The algebra object $\taobj$ is simple.
\end{proposition}

\begin{proof}
Let $\algstuff{I}\neq 0$ be an ideal.
Take $\obstuff{a}=a_{0}+a_{1}\zapar+\cdots+a_{n{-}1}\zapar^{n{-}1}\in\algstuff{I}$ with $a_j \in \K[G/H]$. Let $m \in \llbracket 0, n{-}1 \rrbracket$ be maximal such that $a_{m} \neq 0$. Then
\begin{equation*}
\xx^m \acts a = [m]_{\qpar^{-1}}! (\gat^m.a_m)
\end{equation*}
where $[m]_{\qpar^{-1}}!:=\prod_{j=1}^{m}{\frac{1-\qpar^{-j}}{1-\qpar^{-1}}}$ . Since $m <n$, $[m]_{\qpar^{-1}}! \in \K^{\ast}$, hence $0 \neq \xx^m \acts a \in \algstuff{I} \cap \K[G/H]$.
Now, let $\obstuff{u}=u_{0}+u_{1}\gat+\cdots+u_{k{-}1}\gat^{k{-}1}\in\algstuff{I}\cap\K[G/H]$,
with $u_{i}\in\K$ and $u_{k{-}1}\neq 0$. We claim that $u_{k{-}1}\in\algstuff{I}$. If $k=1$, this is immediate, so suppose that $k>1$. We have
\[
\gt\acts\obstuff{u}=\gt\acts(u_{0}+u_{1}\gat+\cdots+u_{k{-}1}\gat^{k{-}1})
=u_{0}+\qpar^{-d}u_{1}\gat+\dots+\qpar^{{-}(k{-}1)d}u_{k{-}1}\gat^{k{-}1}\in\algstuff{I},
\]
since $\gt\acts\algstuff{I}\subset\algstuff{I}$.
This implies that
\[
\gat^{-1}(\obstuff{u}-\gt\acts\obstuff{u})=
(1-\qpar^{-d})u_{1}+\dots+(1-\qpar^{-(k{-}1)d})u_{k{-}1}\gat^{k{-}2}\in\algstuff{I}.
\]
By repeating this argument, we see that $u_{k{-}1}\in\algstuff{I}$.
This shows that $\algstuff{I}$ contains a nonzero scalar, so $\algstuff{I}=\taobj$.
\end{proof}

For any $0\leq i\leq k-1$, let
\[
\obstuff{e}_{i}:=\tfrac{1}{k}{\textstyle\sum_{j=0}^{k-1}}\,\qpar^{-dij}\gat^{j}
\]
be the central idempotent of $\K[G/H]$ corresponding to the character
$\xi_{i}\colon\K[G/H]\to\K^{\ast}$, which is given by
$\xi_{i}(\gat)=\qpar^{di}$.

\begin{lemma}\label{lemma:taft-y-unique}
The element $\zapar$ is the only element of $\taobj$ that satisfies
$\gt\acts\zapar=\qpar^{-1}\zapar$,
$\xx\acts\zapar=1$ and
\begin{gather}\label{eq:2nd-prop}
\zapar={\textstyle\sum_{i=0}^{k{-}1}}\,\obstuff{e}_{i{-}1}\zapar\obstuff{e}_{i}.
\end{gather}
\end{lemma}

\begin{proof}
Note that $\qpar^{d}$ is a primitive
$k^{\mathrm{th}}$ root of unity, so
$\obstuff{e}_{i}\obstuff{e}_{j}=\delta_{i,j}\obstuff{e}_{i}$, for all $i,j=0,\ldots,k-1$,
and $\obstuff{e}_{0}+\obstuff{e}_{1}+\dots+\obstuff{e}_{k{-}1}=1$.
Moreover $\obstuff{e}_{i}\zapar=\zapar\obstuff{e}_{i{+}1}$, so
$\zapar={\textstyle\sum_{i=0}^{k-1}}\,\obstuff{e}_{i{-}1}\zapar\obstuff{e}_{i}$.
Using \eqref{eq:x-action}, the claim follows.
\end{proof}

\begin{proposition}\label{prop:taft-isos-A}
The algebra objects $\taobj$ and $\taobj[d^{\prime},\lambda^{\prime}]$ are isomorphic  if and only if
$d=d^{\prime}$ and $\lambda=\lambda^{\prime}$.
\end{proposition}

\begin{proof}
Suppose that $\phi\colon\taobj\to\taobj[d^{\prime},\lambda^{\prime}]$ is an
isomorphism of algebra objects. Then \fullref{lemma:taft-simple-object}
immediately implies that $d=d^{\prime}$. Thanks to \fullref{lemma:taft-y-unique}, we know that $\zapar$ is the only element of $\taobj$ and $\taobj[d,\lambda^{\prime}]$ that satisfies $\gt\acts\zapar=\qpar^{-1}\zapar$,
$\xx\acts\zapar=1$ and \eqref{eq:2nd-prop}. This implies that $\phi(\zapar)=\zapar$, since all these properties have to be preserved by $\phi$. We get that $\lambda=\lambda^{\prime}$, because
$\lambda\cdot 1=\phi(\zapar^{n})=\phi(\zapar)^{n}=\lambda^{\prime}\cdot 1$. The other direction is trivial.
\end{proof}

Recall from \fullref{subsection:taft-algebra-objects} that
$\morstuff{U}\colon\oneMod_{\tmod}(\taobj)\to\tmod$
denotes the forgetful functor, which is right
adjoint to the free functor defined by
$\morstuff{F}(\tmodule)=\tmodule\otimes\taobj$.

\begin{lemma}\label{lemma:decomp-A}
We have
\[
\morstuff{U}(\taobj)\cong{\textstyle\bigoplus_{i=0}^{k-1}}\,\tprmod[1{-}di].
\]
In particular, $\morstuff{U}(\taobj)$ is a projective $\taft$-module.
\end{lemma}

\begin{proof}
Consider the following (alternative) basis of $\taobj$:
\[
\{
\qpar^{-dij}\gat^{i}(\xx^{j}\acts\zapar^{n{-}1})\mid 0\leq i\leq k-1,\; 0\leq j\leq n-1
\}.
\]
For each $0\leq i\leq k-1$, let
\[
\taobj{}(i)=\K\{
\qpar^{-dij}\gat^{i}(\xx^{j}\acts\zapar^{n{-}1})\mid 0\leq j\leq n-1
\}.
\]
Each $\taobj(i)$ is $\taft$-invariant,
and  $\taobj(i)\cong\tprmod[1{-}di]$.
This gives the claimed decomposition and the statement
then follows from \fullref{lemma:taft-prinjective}.
\end{proof}

\begin{remark}\label{remark:taft-eigenspace}
To avoid possible confusion arising from conventions, we note that as a
$\taft$-module we have
\[
\mathrm{hd}(\taobj(i))\cong\tsimple[1-di], \quad \mathrm{soc}(\taobj(i))\cong\tsimple[-di].
\]
\end{remark}

\begin{theorem}\label{theorem:semisimplicity}
For each simple object $\tmodule$ of $\oneMod_{\tmod}(\taobj)$, there exists a unique $r \in \llbracket 0, d\smin 1 \rrbracket$, such that  $\tmodule \cong \tsimple\otimes\taobj$.
\end{theorem}

\begin{proof}
First of all, $\tsimple\otimes\taobj$ is a projective object in
$\oneMod_{\tmod}(\taobj)$ for all $r \in \Z/n\Z$.
Indeed, thanks to \fullref{lemma:decomp-A}, we know that $\morstuff{U}(\taobj)$ is a projective $\taft$-module. Since the forgetful functor $\morstuff{U}$ is exact, its left adjoint $\morstuff{F}$ send projective to projective. Hence ${\morstuff{F}(\morstuff{U}(\taobj))=\morstuff{U}(\taobj) \otimes \taobj}$ is projective in $\oneMod_{\tmod}(\taobj)$. Moreover, the multiplication of $\taobj$ is split by $\morstuff{u}\otimes \mathrm{id}_{\taobj}$, where we recall that $\morstuff{u}$ denotes the unit morphism of the algebra object $\taobj$. Therefore $\taobj$ is a direct summand of the projective right $\taobj$-module $\morstuff{U}(\taobj) \otimes \taobj$, and hence $\taobj$ is projective in $\oneMod_{\tmod}(\taobj)$.

Finally, by \fullref{corollary:taft-tensor-exact} the endofunctor $\tsimple\otimes\placeholder$ of $\tmod$  is invertible, hence gives an autoequivalence of $\oneMod_{\tmod}(\taobj)$, with inverse $\tsimple[-r]\otimes\placeholder$. It therefore preserves projective objects. Thus $\tsimple \otimes \taobj$ is projective for every $r$.

This implies that
\[
\Hom_{\taobj}(\tsimple\otimes\taobj,\placeholder)\colon
\oneMod_{\tmod}(\taobj)\to\mathrm{Vec}
\]
is an exact functor to the category of vector
spaces, since $\taobj$ is projective, see \fullref{lemma:decomp-A}, and
$\tsimple[-r]\otimes\placeholder$, being a right adjoint, is exact. The claim follows.

Furthermore, we get
\begin{gather*}
\Hom_{\taobj}(\tsimple\otimes\taobj,\tsimple[r{+}d]\otimes\taobj)
\cong
\Hom_{\taobj}(\taobj,\tsimple[d]\otimes\taobj)\cong
\Hom_{\taobj}(\taobj,\taobj),
\end{gather*}
because $\tsimple[-r]\otimes\tsimple[r{+}d]\cong\tsimple[d]$ holds in $\tmod$,
see \fullref{proposition:taft-tensor}, and
$\tsimple[d]\otimes\taobj\cong\taobj$ holds in $\oneMod_{\tmod}(\taobj)$ (the
latter isomorphism
is determined by $1\otimes1\mapsto\gat^{-1}$ and
the requirement that it is a morphism of right $\taobj$-modules).
This implies that, for all $r \in \Z/n\Z$, we have an isomorphism
$\tsimple\otimes\taobj\cong\tsimple[r{+}d]\otimes\taobj$ in
$\oneMod_{\tmod}(\taobj)$.
\medskip

Second, for every $\tmodule \in \tmod$, the free $\taobj$-module $\tmodule \otimes \taobj$ is isomorphic to a direct sum of objects of the form
$\tsimple\otimes\taobj$, for $0\leq r\leq d-1$.
It is enough to show this for $\tinde$. We have a short exact sequence:
\begin{equation*}
0 \to \tinde[\ell-1,r+1] \to \tinde \to \tsimple \to 0.
\end{equation*}

Since the tensor product is biexact in a tensor category \cite[Proposition 4.2.1]{EGNO}, we get
\begin{equation*}
0 \to \tinde[\ell-1,r+1] \otimes \taobj \to \tinde \otimes \taobj \to \tsimple \otimes \taobj \to 0.
\end{equation*}

Moreover, we know that $\tsimple \otimes \taobj$ is projective, hence this sequence splits:
\begin{equation*}
\tinde \otimes \taobj \simeq \tsimple \otimes \taobj \oplus \tinde[\ell-1,r+1] \otimes \taobj.
\end{equation*}

Induction gives $\tinde \otimes \taobj \simeq \bigoplus_{j=0}^{\ell-1}{\tsimple[r+j] \otimes \taobj}$. This implies that the original $\tinde\otimes\taobj$ also decomposes as claimed.

Next, by \fullref{proposition:taft-tensor} and \fullref{lemma:decomp-A}, we have for any $0\leq r^{\prime}\leq d-1$ that
\[
\morstuff{U}(\tsimple[r^{\prime}]\otimes\taobj)
\cong{\textstyle\bigoplus_{i=0}^{k{-}1}}\,\tprmod[r^{\prime}{+}1{-}di].
\]
This gives
\begin{gather*}
\Hom_{\taobj}(\tsimple\otimes\taobj,\tsimple[r^{\prime}]\otimes\taobj)
\cong
\Hom_{\tmod}(\tsimple,\morstuff{U}(\tsimple[r^{\prime}]\otimes\taobj))\cong
\bigoplus_{i=0}^{k{-}1}\Hom_{\tmod}(\tsimple,\tprmod[r^{\prime}{+}1{-}di]),
\end{gather*}
using the adjointness of the functors $\morstuff{F}$ and $\morstuff{U}$ again.
Recall that $\mathrm{soc}(\tprmod[r^{\prime}{+}1{-}di])\cong\tsimple[r^{\prime}{-}di]$, for any $0\leq i\leq k-1$, so
\[
\Hom_{\tmod}(\tsimple,\tprmod[r^{\prime}{+}1{-}di])\cong\Hom_{\tmod}(\tsimple,\tsimple[r^{\prime}{-}di]).
\]
If $0\leq r,r^{\prime}\leq d-1$ and $0\leq i\leq k-1$, then $r \equiv r^{\prime} \smin di \mod n$ if
and only if $r = r^{\prime}$ and $i=0$. Hence,
by Schur's lemma, we get further that
\begin{equation} \label{eq:hom-comp}
\Hom_{\taobj}(\tsimple\otimes\taobj,\tsimple[r^{\prime}]\otimes \taobj)\cong
\begin{cases}
\K, &\text{if }r=r^{\prime},\\
0, &\text{else}.
\end{cases}
\end{equation}

This finishes the proof that $\{\tsimple\otimes\taobj\mid 0\leq r\leq d-1\}$ forms a complete and
irredundant set of simple objects in $\oneMod_{\tmod}(\taobj)$. Indeed, the projective object $\obstuff{P}:= \bigoplus_{r=0}^{d{-}1}{\tsimple[r] \otimes \taobj}$ is a (projective) generator since we know that every free $\taobj$-module is a direct sum of objects $\tsimple\otimes \taobj$ and that every $\taobj$-module is a quotient of a free $\taobj$-module. %general to algebra objects
Moreover, \eqref{eq:hom-comp} gives that $\End_{\taobj}(\obstuff{P}) \cong \K^d$. Hence $\oneMod_{\tmod}(\taobj) \simeq \End_{\taobj}(\obstuff{P})^{\mathrm{op}}\text{-}\catstuff{M}\mathrm{od}_{\mathrm{fd}}$. The category $\oneMod_{\tmod}(\taobj)$  is thus semisimple with pairwise nonisomorphic simple objects $\tsimple[r] \otimes \taobj$, with $r \in \llbracket 0, d \smin 1 \rrbracket$.
\end{proof}

%Since $\tsimple\otimes\taobj$ is projective, cf. Claim 1 in \fullref{theorem:semisimplicity},
The proof of \fullref{theorem:semisimplicity} gives in particular the following corollary.

\begin{corollary}\label{corollary:taft-module-cat}
$\oneMod_{\tmod}(\taobj)$ is a semisimple category.
\end{corollary}

\subsection{Module categories of $\tmod$ with bottom cell apex}\label{subsection:taft-simple-algebra-objects-bottom}

%If $\obstuff{A}$ is an algebra object for the bottom cell, then by definition $\tprmod \otimes \obstuff{A}=0$ for $r \in \llbracket 0, n-1\rrbracket$. Tensoring $\tprmod[0] \twoheadrightarrow \tsimple[0]$ with $\obstuff{A}$ gives $0 \twoheadrightarrow A$. Hence \cite[Theorem 3.17]{EO04} implies that there are no exact categorical representations for the bottom cell. Moreover, \TODO: check + add this argument:  I think exactness should actually be automatic here. A simple categorical representation can be written as mod(A), with A simple from the right and hence simple as an algebra object. Coulembier–Stroiński–Zorman, Theorem 7.1, now shows that every simple algebra object in a finite tensor category is exact. In our particular case, one could also deduce this from Skryabin. \\

Let $\catstuff{M}$ be a module category of $\tmod$ with bottom cell apex. Having $\cell[b]$ as apex implies that for every object $\obstuff{M} \in \catstuff{M}$, $\tprmod \otimes \obstuff{M}=0$ for $r \in \llbracket 0, n-1\rrbracket$. Since $\catstuff{M}$ is a module category, tensoring $\tprmod[0] \twoheadrightarrow \tsimple[0]$ with $\obstuff{M}$ gives $0 \twoheadrightarrow \obstuff{M}$, since the action of a module category is right exact in the $\tmod$ variable. Hence $\catstuff{M}=0$.
This rules out nonzero abelian module categories, cf. \cite[Definition 7.3.1]{EGNO}. To construct categorical representations with bottom cell apex, we have to look for finitary $2$-representations.

\begin{remark}
Finitary representations need not be abelian. Moreover, the semisimplification functor is not right exact \cite[Definition 2.5]{EO22}, so that the preceding argument does not apply to what will follow in \fullref{section:taft-2rep-second}.
\end{remark}

\section{Taft \texorpdfstring{$2$}{2}-representation theory with apex $\cell[u]$}\label{section:taft-2rep-second}
Recall from \fullref{subsection:taft-cells} that we identify monoidal categories with its delooping. In this section we will study the $2$-representations associated with $\taobj$.

\begin{definition}\label{definition:2-reps}
A (left) $2$-representation of a monoidal category $\catstuff{C}$ is
a category $\catstuff{M}$ together with
a (strict) monoidal functor $\rho: \catstuff{C} \to \mathrm{End}(\catstuff{M})$ where $\mathrm{End}(\catstuff{M})$ is the monoidal category of endofunctors of $\catstuff{M}$ with composition as monoidal product and identity functor as unit, see \cite[Example 2.3.12]{EGNO}. To simplify the notation, the action functor $\rho$ will usually be denoted by $\otimes$, so $\obstuff{X} \otimes \obstuff{M} := \rho(\obstuff{X})(\obstuff{M})$ for $(\obstuff{X}, \obstuff{M}) \in \catstuff{C} \times \catstuff{M}$. Let $(\catstuff{M}, \rho)$ be a $2$-representation. We will say that this $2$-representation is
\begin{itemize}
\item \textit{finitary} if $\catstuff{M}$ is $\mathrm{Hom}$-finite ($\mathrm{dim}_{\K}\left(\mathrm{Hom}(\obstuff{X},\obstuff{Y})\right)$ is finite, for all $\obstuff{X}, \obstuff{Y} \in \catstuff{M}$), idempotent-complete and has finitely many isomorphism classes of indecomposable objects.
\item \textit{simple} if it does not have any nonzero, proper $\catstuff{C}$-stable ideals.
\end{itemize}

A morphism of $2$-representations is a functor $\morstuff{F}:\catstuff{M} \to \catstuff{N}$ equipped with a family of isomorphisms $\morstuff{\Phi}_{\obstuff{X},\obstuff{M}}:\morstuff{F}(\obstuff{X} \otimes \obstuff{M}) \iso \obstuff{X} \otimes \morstuff{F}(\obstuff{M})$ natural in $\obstuff{X} \in \catstuff{C}$ and in $\obstuff{M} \in \catstuff{M}$ such that
\begin{equation*}
\begin{cases}
\morstuff{\Phi}_{\obstuff{X} \otimes \obstuff{Y},\obstuff{M}}= (\mathrm{id}_{\obstuff{X}} \otimes \morstuff{\Phi}_{\obstuff{Y},\obstuff{M}})\circ \morstuff{\Phi}_{\obstuff{X},\obstuff{Y}\otimes \obstuff{M}},\\
\morstuff{\Phi}_{1,\obstuff{M}}=\mathrm{id}_{\morstuff{F}(\obstuff{M})}.
\end{cases}
\end{equation*}

\end{definition}

In this article, all $2$-representations are assumed to be finitary.

\begin{remark}
The reader may wonder why do not call these $2$-representations simple transitive. As observed in \cite{Mie26} just before Definition $11$, every simple finitary $2$-representation is transitive.
\end{remark}

\begin{example}\label{example:2-reps-algebra-objects}
Let $\taobj[]$ be an algebra object in a monoidal
category $\catstuff{C}$. Then the category
of its right modules $\oneMod_{\catstuff{C}}(\taobj[])$
is a left $2$-representation of $\catstuff{C}$
where $\obstuff{X}\in\catstuff{C}$ acts as $\obstuff{X}\otimes\placeholder$.
\end{example}

Let us denote by $\Gros(\catstuff{C})$, the split Grothendieck group of a category $\catstuff{C}$.
We observe that $\Gros(\oneMod_{\tmod}(\taobj))$ of
$\oneMod_{\tmod}(\taobj)$ is an $\N$-module
whose $\N$-basis is given by the Grothendieck classes
\[
\{[\tsimple\otimes\taobj]\mid 0\leq r\leq d-1\}.
\]

Recall that we have defined $\taftm(d)$ in \fullref{subsection:taft-n-modules}.
\begin{proposition}\label{proposition:taft-decat}
We have $\Gros(\oneMod_{\tmod}(\taobj))\cong_{\N}\taftm(d)$ as $\taftp$-modules.
\end{proposition}

\begin{proof}
It is enough to compute the
action of $\tsimple[1]$ and $\tinde[2,0]$ on
the set of simple objects identified in \fullref{theorem:semisimplicity}. That we get the same action as for $\taftm(d)$, follows from \fullref{proposition:taft-tensor}.
\end{proof}

Note that the parameter $\lambda$ does not play any role
on the decategorified level. We will now see how it shows up on the categorified level.

%A $2$-representation $\catstuff{M}$ of a tensor
%category $\catstuff{C}$ is assumed to be
%additionally abelian and locally finite.

We recall the definition from \cite[Definition 7.5.1]{EGNO} of an exact module category that we extend to this framework.

\begin{definition}
Let $\catstuff{C}$ be a tensor category. An abelian $2$-representation $\catstuff{M}$ is \textit{exact},
if for every projective object $\tmodule[P]\in\catstuff{C}$ and
every object $\tmodule[N]\in\catstuff{M}$ the object $\tmodule[P]\otimes\tmodule[N]$
is projective in $\catstuff{M}$.
\end{definition}

\begin{remark}\label{remark:taft-exact}
Let $\catstuff{M}$ be $2$-representation of a
tensor category $\catstuff{C}$. Then, for
any object $\tmodule[M]\in\catstuff{C}$ and any projective
object $\tmodule[Q]\in\catstuff{M}$,
the object $\tmodule[M]\otimes\tmodule[Q]$ is always projective in $\catstuff{M}$,
cf. \cite[Exercise 7.5.2]{EGNO}.
\end{remark}

\begin{example}\label{example:taft-exact}
If the unit object in $\catstuff{C}$ is projective,
then the exact $2$-representations of $\catstuff{C}$
are precisely the semisimple $2$-representations of $\catstuff{C}$,
since any of their objects has to be projective.
\end{example}

\begin{proposition}\label{proposition:taft-module-cat}
The $2$-representation
$\oneMod_{\tmod}(\taobj)$ of $\tmod$
is simple.
\end{proposition}

\begin{proof}
Since $\oneMod_{\tmod}(\taobj)$ is semisimple by
\fullref{corollary:taft-module-cat}, and $\tsimple[1]$ acts transitively on the simples, it is a simple  $2$-representation of $\tmod$. Indeed, any nonzero ideal contain the identity morphism of some simple. Transitivity then implies that it contains the identity morphism of all simple objects.
%while exactness also follows directly from the semisimplicity, since any object of a semisimple abelian category is projective.
\end{proof}

\begin{theorem}\label{theorem:taft-module-cat}
The $2$-representations
$\oneMod_{\tmod}(\taobj)$ and $\oneMod_{\tmod}(\taobj[d^{\prime},\lambda^{\prime}])$
are equivalent if and only if ${d=d^{\prime} \text{ and } \lambda=\lambda^{\prime}}$.
\end{theorem}

\begin{proof}
As before, one direction is trivial, so
suppose that $\oneMod_{\tmod}(\taobj)\cong\oneMod_{\tmod}(\taobj[d^{\prime},\lambda^{\prime}])$
as $2$-representations of $\tmod$. Then we must have $d=d^{\prime}$, because the
two categories have to have the same number of simples up to isomorphism, cf. \fullref{theorem:semisimplicity}.

Moreover, adjointness of $\morstuff{F}$ and $\morstuff{U}$ gives a natural isomorphism
\[
\Hom_{\taobj}(\tsimple\otimes\taobj,\taobj)\cong
\Hom_{\tmod}(\tsimple,\taobj).
\]
This implies that
\[
\underline{\End}(\taobj)\cong\taobj
\]
as algebra objects in $\tmod$, where the underline notation means the internal Hom, cf. \cite[Section 7.9]{EGNO},
which is the representing object of the functor
$\Hom_{\taobj}(\placeholder\otimes\taobj,\taobj)$.

In fact, for any $0\leq r\leq d-1$, we also have
\begin{gather}\label{eq:equiv-2reps}
\underline{\End}(\tsimple[r]\otimes\taobj)
\cong\taobj,
\end{gather}
because
\begin{comment}
\begin{gather*}
\Hom_{\taobj}(\tsimple[r^{\prime}]\otimes\tsimple\otimes\taobj,\tsimple\otimes\taobj)\cong
\Hom_{\taobj}(\tsimple[{-}r]\otimes\tsimple[r^{\prime}]\otimes\tsimple\otimes\taobj,\taobj)\cong
\\
\Hom_{\taobj}(\tsimple[{-}r{+}r^{\prime}{+}r]\otimes\taobj,\taobj)\cong
\Hom_{\taobj}(\tsimple[r^{\prime}]\otimes\taobj,\taobj).
\end{gather*}
\end{comment}
\begin{gather*}
\Hom_{\taobj}(\tmodule\otimes\tsimple\otimes\taobj,\tsimple\otimes\taobj)\cong
\Hom_{\tmod}(\tmodule\otimes\tsimple,\tsimple\otimes\taobj)
\\
\hspace*{5.9cm} \cong \Hom_{\tmod}(\tmodule,\tsimple\otimes\taobj\otimes \tsimple[{-}r]).
\end{gather*}
\
In other words, for every simple object $\tsimple\in\oneMod_{\tmod}(\taobj)$,
we have $\underline{\End}(\tsimple)\cong\taobj$ in $\tmod$ since the morphism $\tsimple\otimes\taobj\otimes \tsimple[{-}r]\iso \taobj, \gat \mapsto \gat, \zapar \mapsto \qpar^r\zapar$ is an isomorphism as algebra objects in $\tmod$.
Similarly, for every simple object
$\tsimple[r^{\prime}]\in\oneMod_{\tmod}(\taobj[d,\lambda^{\prime}])$, we have
$\underline{\End}(\tsimple[r^{\prime}])\cong\taobj[d,\lambda^{\prime}]$
in $\tmod$.

Moreover, the equivalence in \eqref{eq:equiv-2reps} sends simples to simples and
commutes with the action of $\tmod$. So
there exists a $0\leq r^{\prime\prime}\leq d-1$ such that the equivalence
sends $\tsimple\otimes\taobj$ to $\tsimple\otimes\tsimple[r^{\prime\prime}]\otimes\taobj[d^{\prime},\lambda^{\prime}]$,
for all $0\leq r\leq d-1$. Therefore, we get
\[
\Hom_{\taobj}(\tsimple\otimes \taobj,\taobj)\cong
\Hom_{\taobj[d,\lambda^{\prime}]}(\tsimple\otimes\tsimple[r^{\prime\prime}]\otimes\taobj[d,\lb^{\prime}],\tsimple[r^{\prime\prime}]\otimes\taobj[d,\lambda^{\prime}])
\]
for all $0\leq r\leq d-1$. This implies that
\[
\taobj\cong\underline{\End}\left(\taobj\right)\cong
\underline{\End}\left(\tsimple[r^{\prime\prime}]\otimes\taobj[d,\lambda^{\prime}]\right)\cong\taobj[d,\lambda^{\prime}]
\]
as algebra objects of $\tmod$.
Hence, the statement follows from \fullref{prop:taft-isos-A}.
\end{proof}

This corollary is then a direct consequence of \fullref{theorem:taft-module-cat}.

\begin{corollary}\label{corollary:taft-module-cat-many}
There are infinitely many, pairwise non-equivalent,
exact, simple $2$-repre\-sentations
of the category $\tmod$.
\end{corollary}

\section{Taft \texorpdfstring{$2$}{2}-representation theory with apex $\cell[b]$}\label{section:taft-2rep-second}

Recall that we identify every monoidal category $\catstuff{C}$ with its delooping, i.e. with the $2$-category with a unique object $*$, hom-category $\mathrm{Hom}(*,*)=\catstuff{C}$ and horizontal composition given by the monoidal product in $\catstuff{C}$.

\subsection{Stable category and semisimplification of $\tmod$}\label{subsec:stable-cat}

\begin{definition}
Denote by $\tpid$ the categorical ideal of morphisms that factor through projective objects i.e. $\tpid(\obstuff{M},\obstuff{N}) := \{\obstuff{M} \to \obstuff{P} \to \obstuff{N} \mid \obstuff{P} \text{ projective}\}$. Since tensoring a $\taft$-module by a projective $\taft$-module gives a projective $\taft$-module (cf. \fullref{proposition:taft-tensor}), $\tpid$ defines a tensor ideal of $\tmod$. We consider the \textit{stable category} of $\tpid$ denoted by $\stmod$ which is the monoidal category:
\begin{center}
$\bullet$ $\mathrm{Obj}(\stmod) := \mathrm{Obj}(\tmod)$,
\qquad $\bullet$ $\stmod(\obstuff{M}, \obstuff{N}) := \tmod(\obstuff{M}, \obstuff{N})/\tpid(\obstuff{M},\obstuff{N})$.
\end{center}
\end{definition}

By definition, every projective module acts by zero on a transitive $2$-representation with apex $\cell[b]$, hence it factors through $\stmod$.
Recall that $\tmod$ is equipped with a natural pivotal structure cf. \fullref{proposition:taft-fiat}. A direct computation gives that:
\begin{equation}
\label{computation:dim}
\mathrm{dim}^L(\tinde[l,r])=\qpar^r\frac{1-\qpar^l}{1-\qpar}, \qquad \mathrm{dim}^R(\tinde[l,r])=\qpar^{1-l-r}\frac{1-\qpar^l}{1-\qpar}.
\end{equation}
Hence, $\mathrm{dim}^L(\tinde[l,r])=0 \iff \mathrm{dim}^R(\tinde[l,r])=0 \iff l=n$.
The categorical ideal of negligible morphisms $\catstuff{N}$ defined by $\catstuff{N}(\obstuff{M}, \obstuff{N}):=\{f: \obstuff{M} \to \obstuff{N} \mid \mathrm{Tr}(fg)=0, \forall g: \obstuff{N} \to \obstuff{M}\}$, for every indecomposable objects $\obstuff{M},\obstuff{N}$ is then a tensor ideal.

\begin{lemma}
For $n \geq 3$, the tensor ideal of negligible morphisms is a nonzero and proper ideal of $\stmod$.
\end{lemma}
\begin{proof}
The simple module $\tsimple[j+1]$ is the socle of $\tinde[2,j]$, which is different from $\tsimple[j]$, its head. Thus $\oneHom_{\stmod}(\tinde[2,j],\tsimple[j+1]) = 0$. If $\iota: \tsimple[j+1] \hookrightarrow \tinde[2,j]$ denotes the canonical inclusion, we have $\mathrm{Tr}(g\iota)=0$ for every $g: \tinde[2,j] \to \tsimple[j+1]$, hence $\iota \in \catstuff{N}(\tsimple[j+1],\tinde[2,j])$. This ideal is moreover proper since $\mathrm{id}_{\tsimple[0]} \notin \catstuff{N}$.
\end{proof}

Since our goal is to construct simple $2$-representations, the stable category is still "too big". Indeed, the ideal of negligible morphisms is contained in its categorical radical and is nilpotent. Its image in a simple finitary $2$-representation is then a nilpotent stable ideal, hence must be $0$. To that end let us consider the semisimplification of $\tmod$. We recall the definition of the semisimplification of a pivotal tensor category as constructed in \cite[\S 2.3]{EO22}.

\begin{definition}
Let $\catstuff{C}$ be a pivotal tensor category such that $\mathrm{dim}^L(\obstuff{M})=0$ if and only if $\mathrm{dim}^R(\obstuff{M})=0$ for any indecomposable object $\obstuff{M}$ of $\catstuff{C}$.  We can then define the \textit{semisimplification}  $\overline{\catstuff{C}} := \catstuff{C}/\catstuff{N}$.
\end{definition}

 Thus, \cite[Proposition 2.4]{EO22} gives that the category $\sstmod$ is a well defined pivotal fusion category.

Note that the semisimplification functor $\pi: \tmod \to \sstmod$ factors through $\stmod$. Indeed, \eqref{computation:dim} gives that the zero-dimensional indecomposable modules are precisely the projective covers $\tinde[n,r]$. This implies that $\catstuff{P} \subset \catstuff{N}$ and that there exists a canonical functor $\sigma: \stmod \to \sstmod$.

\begin{lemma}
The functor $\sigma$ induces an isomorphism of $\N$-based rings ${\Gros(\sigma) : \Gros(\stmod) \iso \Gros(\sstmod)}$.
\end{lemma}
\begin{proof}
The set $\{\tinde[l,r] \mid l \in \llbracket 1, n-1 \rrbracket, r \in \llbracket 0, n-1 \rrbracket\}$ of nonzero indecomposables of $\stmod$ is by construction sent to the simple objects of $\sstmod$, hence the isomorphism.
\end{proof}

Let $\twmod[a] := \tinde[2a+1, -a],$ for $0 \leq a \leq \lfloor \frac{n-2}{2} \rfloor$ and $\catstuff{W}$ the full subcategory of $\sstmod$ generated by these objects.
%Suppose now that $n=2(\ell+1)$, so that $\catstuff{W}= \langle \twmod[0],\dots,\twmod[\ell] \rangle$.

Take $\xi \in \K$ a such that $\xi^2$ is a primitive $n^{\mathrm{th}}$ root of unity and denote by $\tilde{\catstuff{C}}_{\xi}:=\catstuff{C}_{n-2}(\xi)_{\mathrm{ad}}$ where $\catstuff{C}_{n-2}(\xi)$ denotes the Verlinde category associated with ($\mathfrak{sl}_2$,$\xi$) and $\catstuff{C}_{n-2}(\xi)_{\mathrm{ad}}$ is the adjoint subcategory cf. \cite[Definition 4.14.5]{EGNO}.
% If the reader wonders why this notation is consistent, see \cite[Example 4.14.6]{EGNO}.
 Fix $\xi \in \K$ so that $\xi^2=\qpar$.

\begin{proposition}\label{prop:tensor-equiv}
 The category $\catstuff{W}$ is a pivotal fusion subcategory of $\sstmod$. Moreover, there exists a tensor equivalence $\catstuff{W} \simeq \tilde{\catstuff{C}}_{\xi}$ mapping $\twmod[a] \to \obstuff{V}_{2a}$, where $\obstuff{V}_{2a}$ denotes the simple object of even highest weight $2a$.
\end{proposition}
\begin{proof}
Let us start by computing the fusion rules of $\catstuff{W}$. Modulo projectives, \eqref{eq:tensordecomp} gives:
\begin{equation*}
\twmod[a]\otimes \twmod[b] \cong \bigoplus_{c=|a-b|}^{\mathrm{min}(a+b, n-2-a-b)}{\twmod[c]}
\end{equation*}
Hence, $\catstuff{W}$ is closed under the tensor product and $\twmod[0]=\tsimple[0]$. Finally, \fullref{lemma:taft-duals}, gives that $\twmod[a]^*\cong \tinde[2a+1, 1-(2a+1)+a]=\twmod[a]$. We get that $\catstuff{W}$ is a fusion subcategory. The pivotal structure of $\sstmod$ restricts to $\catstuff{W}$ and \eqref{computation:dim} gives that $\catstuff{W}$ is spherical. We have:
\begin{equation*}
\mathrm{dim}(\twmod[a]) = \qpar^{-a}\frac{1-\qpar^{2a+1}}{1-\qpar} = \frac{\xi^{2a+1} - \xi^{-2a-1}}{\xi - \xi^{-1}} =: [2a+1]_{\xi}.
\end{equation*}
If $K$ is of characteristic $0$, we can then directly apply \cite[Theorem A.3]{EO22} when $n$ is even and bigger than $6$ and \cite[Remark A.4(iii)]{EO22} when $n$ is odd and bigger than $3$ to obtain the tensor equivalence. Indeed $\catstuff{W}$ has the correct decategorification. If $n=2(\ell+1)$:

\medskip
\begin{equation*}
\twmod[1]\otimes \twmod[b] \cong
\begin{cases}
\twmod[b-1] \oplus \twmod[b] \oplus \twmod[b+1], &\text{ if } 1 \leq b \leq \ell -1,\\
\twmod[\ell-1], &\text{ if } b=\ell.
\end{cases}
\end{equation*}
If $n=2\ell+1$
\begin{equation*}
\twmod[1]\otimes \twmod[b] \cong
\begin{cases}
\twmod[b-1] \oplus \twmod[b] \oplus \twmod[b+1], &\text{ if } 1 \leq b \leq \ell -2,\\
\twmod[\ell-2] \oplus \twmod[\ell-1], &\text{ if } b=\ell-1.
\end{cases}
\end{equation*}
\medskip

For $n=2,3$, since the only simple object is $\twmod[0]$ which is the trivial module, we directly get that $\catstuff{W} \simeq \mathrm{Vec}\simeq i\tilde{\catstuff{C}}_{\xi}$. For $n=4$, we now have two simple modules. Coming back to the bases element of the nontrivial module, we can show that the associator class in $\mathrm{H}^3(\Z/2\Z,\K^\ast)$ is trivial, hence the desired equivalence. Finally, the remaining case is $n=6$. In that case there are three fusion categories with these fusion rules \cite[Remark 6.6]{EGO04}. Thankfully, to distinguish these fusion categories, it is enough to compute the Frobenius-Schur indicator \cite[Corollary 3.3]{Tuck19}. In our case it can be done explicitly using the basis of the $3$-dimensional simple module. We thus get that $\catstuff{W} \simeq \mathrm{Rep}(\mathfrak{S}_3)$, where $\mathfrak{S}_3$ denote the symmetric group on $3$ letters.

Suppose now that $\K$ has positive characteristic. Note that $\mathrm{dim}(\catstuff{W}) = \frac{n}{2-\qpar-\qpar^{-1}} \neq 0 \in \K$ when $n>2$, and for $n=2$, we get $\mathrm{dim}(\catstuff{W})=1$, since in that case $\catstuff{W} \simeq \mathrm{Vec}$. Hence \cite[Theorem 9.3]{ENO5} can be applied to obtain a unique lift to the ring of Witt vectors of $\K$. The fusion category $\catstuff{W}$ is moreover separable thanks to \cite[Theorem 2.6.7]{DSS20}. Combining this with the preceding result in characteristic $0$ and \cite[Theorem 4.6]{E18} gives the desired tensor equivalence.
\end{proof}

\subsection{Simple $2$-representations}\label{subsection:taft-2-reps-struct}

In this section let us review the general structural results obtained in \cite{MMMTZ23} that will reduce the study to the one of regular exact, indecomposable $\sstmod$-module categories.

\begin{proposition}\label{proposition:bij-bottom-cell-rep-sstmod-rep}
We have the following bijection:
\begin{equation*}
\left\{ \begin{gathered}\text{finitary simple 2-representations}\\ \text{ of }\tmod \text{ with apex } \cell[b] \end{gathered}\right \} \simeq \left\{\text{finitary simple 2-representations of } \sstmod \right\}
\end{equation*}
\end{proposition}
\begin{proof}
Take $\catstuff{M}$ a finitary simple $2$-representation of $\tmod$ with apex $\cell[b]$. Having $\cell[b]$ as apex implies that $\catstuff{M}$ is a finitary $2$-representation of $\stmod$. Moreover, the proper ideal $\catstuff{N}\cap \stmod$ is a nilpotent. Its image is a stable ideal. Since $\catstuff{M}$ is supposed to be simple, this implies that this image has to be zero, since the whole identity ideal cannot be nilpotent. Conversely, every finitary simple $2$-representation of $\sstmod$, is naturally a finitary simple $2$-representation of $\tmod$ and annihilates the cell $\cell[u]$, hence has apex $\cell[b]$.
\end{proof}

\begin{proposition}
Every finitary simple $2$-representation $\catstuff{M}$ of $\sstmod$ is semisimple and is equivalent to $\oneMod_{\sstmod}\left(\underline{\mathrm{End}}(\obstuff{N})\right)$ for $\obstuff{N}$ a simple object of $\catstuff{M}$.
%Moreover, ${\obstuff{A}_{\obstuff{N}}:=\underline{\mathrm{End}}(\obstuff{N}) \in \catstuff{M}}$ is entirely determined by $\Gros(\catstuff{M})$.
\end{proposition}
\begin{proof}
Since $\sstmod$ is a pivotal fusion category, \cite[Proposition 2.16]{MMMTZ23} gives that $\catstuff{M}$ is thus necessarily semisimple.
This follows directly from \cite[Corollary 4.8, Corollary 4.9]{MMMT19}.
\end{proof}

\begin{remark}
As explained in \fullref{subsection:taft-simple-algebra-objects-bottom}, we are considering finitary $2$-representations since there are no module categories of $\tmod$ with bottom cell apex.
For fusion categories, the notions of finitary simple $2$-representation coincide with the notion of indecomposable exact module category cf. \cite[\S 8]{MMMTZ23}, \cite[Remark 13]{Mie26}.
\end{remark}

This proposition reduces the classification problem to the classification of simple algebra objects in $\sstmod$ up to Morita equivalences.

Before starting the case by case discussion, note that when $n>2$, $\sstmod$ can be equipped with a faithful $C_{2n}:=\Z/2n\Z$-grading where $\mathrm{deg}(\tinde[l,r]):= 2r+l-1$. For $n=2$, it appears that $\sstmod$ is equivalent to $\mathrm{Vec}_{\Z/2\Z}$, the monoidal category of finite-dimensional $\Z/2\Z$-graded vector spaces. Hence it is faithufully graded by $C_2$. The equation \eqref{eq:tensordecomp} shows that the tensor structure is compatible with this grading. Moreover, $\sstmod_0=\catstuff{W}$, and $\sstmod_g \neq 0$ for $g \in C_{2n}$. Indeed, if $g \equiv 0 [2]$, then $\tinde[1,g/2] \in \sstmod_g$ else $\tinde[2,\frac{g-1}{2}] \in \sstmod_g$. This turns $\sstmod = \bigoplus_{g \in C_{2n}}{\sstmod_g}$ into a faithfully graded pivotal fusion category in the sense of \cite[\S 4.14]{EGNO}. In the language of \cite{MM12}, this means that $\sstmod$ is an extension of $\catstuff{W}$ by $C_{2n}$.

Denote by $\catstuff{Q}$ the full subcategory of $\sstmod$ generated by $\tsimple[1]$. Note that since $\tsimple[1]^n \simeq \tsimple[0]$, we have $\catstuff{Q} \simeq \mathrm{Rep}(\Z/n\Z)$
\begin{definition}
An exact module category $\catstuff{M}$ of $\sstmod$ is called \textit{regular} if the restriction of $\catstuff{M}$ to $\catstuff{Q}$ is isotypic i.e. is equivalent to $\catstuff{V}^{\oplus k}$ for some indecomposable module category $\catstuff{V}$ of $\catstuff{Q}$ and $k \in \Z_{\geq 1}$.
\end{definition}

\begin{proposition}\label{prop:decat-is-ADET}
If $\catstuff{M}$ is an exact indecomposable regular module category of $\sstmod$, then $\Gros(\catstuff{M})$ is an admissible, regular, transitive, NIM $\N$-module with apex $\cell[b]$.
\end{proposition}
\begin{proof}
Consider the basis of $\Gros(\catstuff{M})$ given by the simple objects in $\catstuff{M}$. Semisimplicity of $\catstuff{M}$ implies that it is an $\N$-module. Moreover, transitivity and regularity directly follow from the fact that $\catstuff{M}$ is a regular module category and a simple $2$-representation. To prove transitivity, let us choose representatives $\obstuff{L}_1,\dots,\obstuff{L}_k$ of simple objects of $\catstuff{M}$. Then rigidity of $\sstmod$ gives the following adjunction isomorphism \cite[Proposition 2.10.8]{EGNO}:
\begin{equation*}
\mathrm{Hom}_{\catstuff{M}}(\obstuff{Z}^* \otimes \obstuff{L}_j,\obstuff{L}_i) \simeq \mathrm{Hom}_{\catstuff{M}}(\obstuff{L}_j,\obstuff{Z}\otimes \obstuff{L}_i)
\end{equation*}

 for each object $\obstuff{Z}$ of $\sstmod$. Thus, if we denote by $M_{\obstuff{Z}}:=\big(\mathrm{dim}(\mathrm{Hom}_{\catstuff{M}}(\obstuff{L}_i,\obstuff{Z}\otimes \obstuff{L}_j)\big)_{i,j \in \llbracket 1,k \rrbracket}$, then $M_{\obstuff{Z}^*}=M_{\obstuff{Z}}^{\mathrm{T}}$. Hence we can use the same argument used in \fullref{proposition:taft-small-apex} to prove that $M_{\xpar}$ is diagonalizable. The fact that this $\N$-module is NIM directly comes from rigidity.
\end{proof}

\subsection{Explicit constructions of type $\typeADET$ module categories}\label{subsection:explicit-const-mod-cat}

In this subsection we will construct exact, indecomposable, regular $\sstmod$-module categories for each pair $(\Gamma,d)$ with $d|n$ and $\Gamma$ an $\typeADET$ diagram with Coxeter number equal to $n$. Let $k:=n/d$ and consider the following object of $\sstmod$:
\begin{equation*}
\obstuff{R}_d := \bigoplus_{t=0}^{k-1}{\tsimple[td]}.
\end{equation*}

\begin{proposition}\label{prop:alg-to-mod-cat}
For every simple and separable algebra object $\obstuff{A} \in \catstuff{W}$, $\obstuff{B}_d(\obstuff{A}):=\obstuff{R}_d \otimes \obstuff{A}$ is a simple and separable algebra object in $\sstmod$.

Moreover $\oneMod_{\sstmod}\left(\obstuff{B}_d(\obstuff{A})\right)$ is an exact, indecomposable, regular $\sstmod$-module category.
\end{proposition}
\begin{proof}
Note that by construction $\obstuff{R}_d$ is a central algebra object with multiplication given by $\tsimple[td] \otimes \tsimple[t'd] \simeq \tsimple[(t+t')d]$. Indeed, for every $\taft$-module $\tmodule$, consider ${\gamma_{r,\tmodule}:\tsimple \otimes \tmodule \iso \tmodule \otimes \tsimple}$, ${1 \otimes m \mapsto \gt^{-r}m \otimes 1}$. The morphism $\gamma_{r,\tmodule}$ is $\taft$-linear since $\gt^{-r}$ is grouplike. Moreover, it satisfies the half-braiding coherence cf. \cite[Definition 7.13.1]{EGNO}. The direct sum of these maps gives a half-braiding $\gamma_{\tmodule}:\obstuff{R}_d \otimes \tmodule \iso \tmodule \otimes \obstuff{R}_d$. Compatibility with the multiplication of $\obstuff{R}_d$ follows from $\gt^{{-}(a+b)}=\gt^{{-}a}\gt^{{-}b}$.

If $\obstuff{A}$ is an algebra object, we can define a multiplication on $\obstuff{B}_{d}(\obstuff{A})$ via the multiplication of $\obstuff{R}_d$ and $\obstuff{A}$. Associativity and the unit axiom follow from the one of $\obstuff{R}_d$ and $\obstuff{A}$.

 Let us now show that $ \obstuff{B}_d(\obstuff{A})$ is simple. The summand $\obstuff{B}_t := \tsimple[td] \otimes \obstuff{A}$ is homogeneous of degree $2td$. Take $\algstuff{I}\neq 0$ be an ideal of $\obstuff{B}_d(\obstuff{A})$. Take $t$ such that $\algstuff{I}_t:= \algstuff{I} \cap \obstuff{B}_t \neq 0$. Multiplying by the central invertible summand $\tsimple[-td] \otimes 1_{\obstuff{A}}$ identifies $\obstuff{B}_t$ with $\obstuff{B}_0=\obstuff{A}$. Hence $\algstuff{I}_0 := (\tsimple[-td] \otimes 1_{\obstuff{A}}) \otimes \algstuff{I}_t$ is a nonzero two-sided ideal of $\obstuff{A}$, and is therefore equal to  $\obstuff{A}$, since $\obstuff{A}$ is simple. In particular $\algstuff{I}=\obstuff{B}_d(\obstuff{A})$.

Furthermore, let us show that $\obstuff{B}_d(\obstuff{A})$ is separable. Since $k$ is invertible in $\K$, $\obstuff{R}_d$ is separable. Indeed, a splitting of the multiplication of $\obstuff{R}_d$ is given by
\begin{equation*}
s_{\obstuff{R}_d}(m_{td}):=\frac{1}{k}\sum_{i=0}^{k-1}{m_{id} \otimes m_{(t-i)d}}.
\end{equation*}
A splitting of the multiplication of $\obstuff{B}_d(\obstuff{A})$ is then given by ${(\mathrm{id}_{\obstuff{R}_d} \otimes \gamma_{\obstuff{A}} \otimes \mathrm{id}_{\obstuff{A}}) \circ (s_{\obstuff{R}_d} \otimes s_{\obstuff{A}})}$.

The fact that $\oneMod_{\sstmod}\left(\obstuff{B}_d(\obstuff{A})\right)$ is an exact module category follows from \cite[Remark 7.8.31]{EGNO} and the fact that $\obstuff{B}_d(\obstuff{A})$ is simple implies that the associated module category is indecomposable. Finally, if $C_{2d}:=\langle 2d \rangle \subseteq C_{2n}$, note that $\obstuff{B}_d(\obstuff{A}) \subset (\sstmod)_{C_{2d}} := \bigoplus_{g \in C_{2d}}{\sstmod_g}$. Hence, every simple $\obstuff{B}_d(\obstuff{A})$-module is supported on a coset of $C_{2d}$. Each $\tsimple[1]$-orbit of simple $\obstuff{B}_d(\obstuff{A})$-modules has length divisible by $d$ since $\tsimple[1]$ has degree $2$ and $2$ has order $d$ in $C_{2n}/C_{2d}$. To prove that the length is exactly $d$, it is enough to show that $\tsimple[d] \otimes \tmodule \simeq \obstuff{L}$, for every simple $\obstuff{B}_d(\obstuff{A})$-module $\obstuff{L}$. To prove that, consider the twisted half braiding morphism ${\hat{\gamma}_{d,\tmodule}:\tsimple[d] \otimes \tmodule \to \tmodule \otimes \tsimple[d]}$, ${1 \otimes m \mapsto \xi^{dp}\gt^{-d}m \otimes 1}$, where $\xi^2=\qpar$. If $\iota_d:\tsimple[d] \to \obstuff{B}_d(\obstuff{A})$ denotes the inclusion induced by the unit of $\obstuff{A}$, by construction we have $\morstuff{m}_{\obstuff{B}_d(\obstuff{A})} \circ (\iota_d \otimes \mathrm{id}_{\obstuff{B}_d(\obstuff{A})}) = \morstuff{m}_{\obstuff{B}_d(\obstuff{A})} \circ (\mathrm{id}_{\obstuff{B}_d(\obstuff{A})}\otimes \iota_d) \circ \hat{\gamma}_{d,\obstuff{B}_d(\obstuff{A})}$. We then get that the following composition is a morphism of $\obstuff{B}_d(\obstuff{A})$-modules:
\begin{equation*}
\tsimple[d] \otimes \obstuff{L} \xlongrightarrow{\hat{\gamma}_{d,\obstuff{L}}} \obstuff{L} \otimes \tsimple[d] \xlongrightarrow{\mathrm{id}_{\obstuff{L}}\otimes \iota_d} \obstuff{L} \otimes \obstuff{B}_d(\obstuff{A}) \xlongrightarrow{\acts_{\obstuff{L}}} \obstuff{L}.
\end{equation*}
This morphism is an isomorphism since $\tsimple[d] \subset \obstuff{B}_d(\obstuff{A})$ is invertible. Therefore, every $\tsimple[1]$-orbit generates an indecomposable $\catstuff{Q}$-module category equivalent to $\oneMod_{\catstuff{Q}}\left(\obstuff{R}_d\right)$. Indeed, the restriction of  $\oneMod_{\catstuff{\sstmod}}\left(\obstuff{B}_d(\obstuff{A})\right)$ to $\catstuff{Q}$ decomposes into the full additive subcategories generated by the $\tsimple[1]$-orbits of simple $\obstuff{B}_d(\obstuff{A})$-modules. Since every such orbit has length $d$, with stabilizer $C_{2d}$ all these indecomposable summands are mutually equivalent, thanks to the facts that $C_{2d}$ is cyclic and $\K$ is algebraically closed.
%$\mathrm{H}^2(H_d,\K^*)=0$
\end{proof}

\subsubsection{Type $\typeADE$}
If $\K$ is of characteristic zero, we can simply import the algebra objects constructed in \cite[Theorem 6.1]{Ost03}, thanks to \fullref{prop:tensor-equiv}. The following table lists the respective simple, separable algebra objects of $\catstuff{W}$

\medskip
\begin{center}
\begin{tabular}{|c|c|}
\hline
\text{Type} $\Gamma$ & \text{algebra object} $\obstuff{A}_{\Gamma}$\\
\hline
$\typeA_{n-1}$& $\tsimple[0]$ \\
\hline
$\typeD_{\ell+1}$, with $n=2\ell$ & $\tsimple[0] \oplus \tinde[n-1,\ell+1]$\\
\hline
$\typeE_6$ & $\tsimple[0] \oplus \tinde[7,9]$\\
\hline
$\typeE_7$ & $\tsimple[0] \oplus \tinde[9,14] \oplus \tinde[17,10]$\\
\hline
$\typeE_8$ & $\tsimple[0] \oplus \tinde[11,25] \oplus \tinde[19,21] \oplus \tinde[29,16].$\\
\hline
\end{tabular}
\end{center}
\medskip

The construction of these algebra objects extends over an algebraically closed field of characteristic not dividing $n$. Indeed, the Temperley-Lieb universal property used in the construction holds in that case \cite[Theorem 2.4 and Corollary 2.5]{Ost08}. Applying \fullref{prop:alg-to-mod-cat}, we get an exact, indecomposable, regular $\sstmod$-module category $\oneMod_{\sstmod}\left(\obstuff{B}_d(\obstuff{A}_{\Gamma})\right)$.

\subsubsection{Type $\typeT$}
We have to treat this type separately because it does not admit a categorification as a $\catstuff{W}$-module category. The important fact is that we can however construct such a categorification if we look at $\sstmod$-module categories. To do so, suppose that $n=2\ell+1$ and take $d|n$. As before let $k=n/d$. Let $\JT:=\tinde[n-1,1]$ and $\obstuff{K}_d:=\tsimple[-\frac{d(k-1)}{2}]\otimes \JT$. Since $\JT \otimes \JT \simeq \tsimple[0]$ and $\JT \otimes \tsimple[1] \simeq \tsimple[1] \otimes \JT$, we have that $\obstuff{K}_d \otimes \obstuff{K}_d \simeq \tsimple[d]$. Furthermore, $\obstuff{K}_d^{\otimes k} \simeq \JT$. Hence $\obstuff{K}_d$ has order $2k$.

\begin{lemma}\label{lemma:Kd-struct}
The full subcategory of $\sstmod$ generated by $\obstuff{K}_d$ is tensor equivalent to $\mathrm{Vec}_{\Z/2k\Z}$, the monoidal category of finite-dimensional $\Z/2k\Z$-graded vector spaces.
\end{lemma}
\begin{proof}
Applying \cite[Theorem 2.11.5]{EGNO}, gives that $\langle \obstuff{K}_d \rangle$ is tensor equivalent to $\mathrm{Vec}^{\omega}_{\Z/2k\Z}$ for some $[\omega] \in \mathrm{H}^3(\Z/2k\Z, \K^*)$. To conclude, it is enough to show that $[\omega]=0$. Since $\obstuff{K}_d^{\otimes 2} \simeq \tsimple[d]$ and that $\langle \tsimple[d] \rangle$ is tensor equivalent to $\mathrm{Vec}_{\Z/k\Z}$, we already have that $\mathrm{Res}^{\Z/2k\Z}_{\Z/k\Z}=0$. It remains to show that its restriction to $\langle \JT \rangle$ is also trivial. Indeed, thanks to \cite[Theorem 6.2.2]{Wei94}, we get that $\mathrm{H}^3(\Z/2k\Z, \K^{\ast})=\mu_{2k}$ the group of $2k$-roots of unity in $\K$. Under this identification, restriction to $\Z/k\Z$ sends $\zeta$ to $\zeta^2$ and to $\Z/2\Z$ sends $\zeta$ to $\zeta^k$. If we show that both restrictions are trivial, we get $\zeta^2=\zeta^k=1$. Since $k$ is odd, this implies that $\zeta=1$. To prove that the restriction to $\langle \JT \rangle$ is trivial we will do an explicit computation. Consider
\begin{center}
$\varphi: \begin{array}{ccc}
\JT \otimes \JT& \to& \tsimple[0]\\
m_i \otimes m_j &\mapsto &\begin{cases} a_im &\text{ if }j=n-i,\\
                                      0 &\text{ else} \end{cases}
\end{array}$
\end{center}
where $a_i:=(-1)^{i-1}\qpar^{\frac{i(i-1)}{2}}$. Moreover, consider
\begin{center}
$\psi: \begin{array}{ccc}
\tsimple & \to & \JT \otimes \JT\\
m &\mapsto &\sum_{i=1}^{n-1}{b_im_i\otimes m_{n-i}}
\end{array}$
\end{center}
where $b_i:=(-1)^i\qpar^{\frac{-i(i+1)}{2}}$. Note that $\varphi$ and $\psi$ are $\taft$-linear since $a_{i+1}+\qpar^ia_i=0$ and $b_{i-1}+\qpar^ib_i=0$. By construction, $\varphi\circ\psi=\mathrm{id}_{\tsimple[0]}$.
Now consider $E:=\varphi \otimes \mathrm{id}_{\JT}$ and $F:=(\mathrm{id}_{\JT}\otimes \varphi)$. Since $\JT$ is invertible and $\K$ is algebraically closed, we have $\mathrm{Hom}_{\sstmod}(\JT^{\otimes 3},\JT)\simeq \K$. Hence there exists a unique scalar $\epsilon \in \K$, such that $E=\epsilon F$. This scalar corresponds to the restriction of $[\omega]$ to the subgroup $\Z/2\Z$. To conclude, it is thus enough to show that $\epsilon=1$. Composing $E=\epsilon F$ with $\psi \otimes \mathrm{id}_{\JT}$, we have $E(\psi \otimes \mathrm{id}_{\JT})=\mathrm{id}_{\JT}$. Moreover, for every basis element $m_j$ of $\JT$, $F(\psi \otimes \mathrm{id}_{\JT}))(m_j)=b_ja_{n-j}m_j=m_j$, since $b_ja_{n-j}=1$. Hence $\mathrm{id}_{\JT}=\epsilon \mathrm{id}_{\JT}$ and therefore $\epsilon=1$.
\end{proof}

Consider now $\obstuff{B}_d(\typeT_{\ell}):=\bigoplus_{i=0}^{2k-1}{\obstuff{K}_d^{\otimes i}} \in \sstmod$ and let us construct an algebra structure on this object. For $i \in \llbracket 0, 2k-1 \rrbracket$, let $\iota_i:\obstuff{K}_d^{\otimes i} \to \obstuff{B}_d(\typeT_{\ell})$ denote the inclusion morphism. Moreover, thanks to \fullref{lemma:Kd-struct}, we can choose isomorphisms $c_{i,j}: \obstuff{K}_d^{\otimes i} \otimes \obstuff{K}_d^{\otimes j} \iso \obstuff{K}_d^{\otimes i+j}$, for each $i,j \in \llbracket 0, 2k-1 \rrbracket$ satisfying the coherence axiom
\begin{equation}\label{eq:coherence}
c_{i+j,l} \circ (c_{i,j} \otimes \mathrm{id}) = c_{i,j+l} \circ (\mathrm{id}\otimes c_{j,l}).
\end{equation}
Define the multiplication map componentwise
\begin{equation*}
\mu\circ(\iota_i \otimes \iota_j) := \iota_{i+j} \circ c_{i,j}
\end{equation*}
and the unit to be $\iota_0$.

\begin{proposition}\label{prop:type-T-eir}
The triple $(\obstuff{B}_d(\obstuff{A}_{\typeT_{\ell}}),\mu,\iota_0)$ defines a simple algebra object of $\sstmod$, such that the associated module category is regular. Moreover, if the characteristic of $\K$ is different from $2$, this algebra object is separable.
\end{proposition}
\begin{proof}
The equation \eqref{eq:coherence} gives directly the associativity of $\mu$ and the unit axioms are satisfied by the unit coherence axioms of the tensor equivalence given by \fullref{lemma:Kd-struct}.

Let us now prove that this algebra is simple. Take $\algstuff{I}$ a nonzero (two-sided) ideal of $\obstuff{B}_d(\obstuff{A}_{\typeT_{\ell}})$. Since $\obstuff{K}_d^{\otimes 0}, \obstuff{K}_d,\dots, \obstuff{K}_d^{\otimes 2k-1}$ are pairwise nonisomorphic, invertible, simple objects of $\langle \obstuff{K}_d \rangle$. This implies that $\obstuff{B}_d(\obstuff{A}_{\typeT_{\ell}})$ is multiplicity-free, hence there exists $i \in \llbracket 0, 2k-1 \rrbracket$ such that $\obstuff{K}_d^{\otimes i} \subset \algstuff{I}$. Multiplying by $\obstuff{K}_d^{-\otimes i}$ gives that $\tsimple[0]$ is a subobject of $\algstuff{I}$, hence $\algstuff{I}=\obstuff{B}_d(\obstuff{A}_{\typeT_{\ell}})$.

As in the proof of \fullref{prop:alg-to-mod-cat}, to prove that the associated module category is regular, it is enough to show that every $\tsimple[1]$-orbit of simple objects of $\oneMod_{\sstmod}(\obstuff{B}_d(\obstuff{A}_{\typeT_{\ell}}))$ has length $d$. In this case $\obstuff{B}_d(\obstuff{A}_{\typeT_{\ell}})$ is supported on $C_d:= \langle d \rangle \subseteq C_{2n}$. Since $n$ is odd, $d$ is odd, hence the smallest integer $p$ such that $2p \in C_d$ is $d$, which shows that the length of every $\tsimple[1]$-orbit is divisible by $d$. Furthermore, observe that every simple $\obstuff{B}_d(\obstuff{A}_{\typeT_{\ell}})$-module is isomorphic to $\obstuff{X} \otimes \obstuff{B}_d(\obstuff{A}_{\typeT_{\ell}})$ for some simple $\obstuff{X} \in \sstmod$. By construction, the set of simple objects of $\sstmod$ is $\{\tsimple,\tinde \mid r \in \Z/n\Z, 0 < \ell <n\}$. Moreover, the action of $\langle \obstuff{K}_d \rangle$ on this set is free since $n$ is odd. Hence  if $\obstuff{X}$ is a simple object of $\sstmod$, $\obstuff{X} \otimes \obstuff{B}_d(\obstuff{A}_{\typeT_{\ell}})= \bigoplus_{i=0}^{2k-1}{\obstuff{X}\otimes \obstuff{K}_d^{\otimes i}}$ is a simple $\obstuff{B}_d(\obstuff{A}_{\typeT_{\ell}})$-module since $(\obstuff{X}\otimes \obstuff{K}_d^{\otimes i})_i$ are nonisomorphic simple objects of $\sstmod$ permuted transitively by the tensor powers of $\obstuff{K}_d$. Moreover, every simple $\obstuff{B}_d(\obstuff{A}_{\typeT_{\ell}})$-module occurs in this way by the free-forgetful adjunction $(\mathrm{F},\mathrm{U})$ defined in \fullref{subsection:taft-algebra-objects}.
Since $\tsimple[d] \simeq \obstuff{K}_d \otimes \obstuff{K}_d$, we get that $\morstuff{m}:\tsimple[d] \otimes \obstuff{B}_d(\obstuff{A}_{\typeT_{\ell}}) \to \obstuff{B}_d(\obstuff{A}_{\typeT_{\ell}})$ is an isomorphism of $\obstuff{B}_d(\obstuff{A}_{\typeT_{\ell}})$-modules. This gives that:
\begin{equation*}
\tsimple[d] \otimes \obstuff{X} \otimes \obstuff{B}_d(\obstuff{A}_{\typeT_{\ell}}) \xrightarrow{\gamma_{d,\obstuff{X}}\otimes \mathrm{id}_{\obstuff{B}_d(\obstuff{A}_{\typeT_{\ell}})}} \obstuff{X} \otimes \tsimple[d] \otimes \obstuff{B}_d(\obstuff{A}_{\typeT_{\ell}}) \xrightarrow{\mathrm{id}_{\obstuff{X}} \otimes \morstuff{m}} \obstuff{X} \otimes \obstuff{B}_d(\obstuff{A}_{\typeT_{\ell}})
\end{equation*}
is an isomorphism for every simple object $\obstuff{X} \in \sstmod$ where $\gamma_{d,\obstuff{X}}$ is the half-braiding introduced in the proof of \fullref{prop:alg-to-mod-cat}. Therefore, every $\tsimple[1]$-orbit has length exactly $d$.

Finally, $\obstuff{B}_d(\obstuff{A}_{\typeT_{\ell}})$ is separable since $2k \in \K^*$. Consider the following splitting of $\mu$
\begin{center}
$s:\begin{array}{ccccc}
\obstuff{B}_d(\obstuff{A}_{\typeT_{\ell}}) & \to & \obstuff{B}_d(\obstuff{A}_{\typeT_{\ell}}) \otimes \obstuff{B}_d(\obstuff{A}_{\typeT_{\ell}})\\
(x_j)_{j} & \mapsto & (\frac{1}{2k}\sum_{i=0}^{2k-1}{c^{-1}_{i,j-i}(x_j)})_{j}\\
\end{array}$.
\end{center}
\end{proof}

\begin{remark}
We will see later on that $\oneMod_{\sstmod}(\obstuff{B}_d(\obstuff{A}_{\typeT_{\ell}}))$ is an exact $\sstmod$-module category even when the characteristic of $\K$ is equal to $2$.
Note that
\begin{equation*}
\obstuff{B}_d(\obstuff{A}_{\typeT_{\ell}}) \simeq \bigoplus_{i=0}^{k-1}{\tsimple[id] \oplus \tsimple[id]\otimes \JT} \simeq \obstuff{R}_d \otimes (\tsimple[0] \oplus \JT).
\end{equation*}
Hence define $\obstuff{A}_{\typeT_{\ell}}:=\tsimple[0] \oplus \JT$. It is an algebra object of $\sstmod$ which is not in $\catstuff{W}$ since the degree of $\JT$ is $n$ and not $0$.
\end{remark}

\subsubsection{Equivariantization}
For $\Gamma$ of type $\typeDT$, we will construct $\oneMod_{\sstmod}(\obstuff{B}_d(\obstuff{A}_{\Gamma}))$ as a certain equivariantization of the exact, indecomposable, regular $\sstmod$-module category of type $\typeA$ cf \cite[Definition 2.7.2]{EGNO}. Note that:
\begin{equation*}
{\oneMod_{\sstmod}(\obstuff{B}_d(\obstuff{A}_{\typeA_{n-1}}))=\oneMod_{\sstmod}(\obstuff{R}_d)}.
\end{equation*}
If $n=2\ell$, let $\JD:=\tinde[n-1,\ell+1]$. When we do not need to treat type $\typeD$ and $\typeT$ separately we will just use $\J$ for $\JD$ and $\JT$. Let us also denote by $\obstuff{O}_{\J}:=\obstuff{R}_d \otimes \J$. The multiplications of $\obstuff{B}_d(\obstuff{A}_{\typeD})$ and of $\obstuff{B}_d(\obstuff{A}_{\typeT})$ induce an isomorphism of $\obstuff{R}_d$-bimodules $\theta_{\J}:\obstuff{O}_{\J} \otimes_{\obstuff{R}_d} \obstuff{O}_{\J} \iso \obstuff{R}_d$. We can now define a module autoequivalence:
\begin{center}
$\Phi_{\J}:\begin{array}{ccc}
\oneMod_{\sstmod}(\obstuff{R}_d) &\to & \oneMod_{\sstmod}(\obstuff{R}_d)\\
\algstuff{N} & \mapsto & \algstuff{N}\otimes_{\obstuff{R}_d} \obstuff{O}_{\J}
\end{array}$.
\end{center}
Since $\J\otimes \J \simeq \tsimple[0]$, $\Phi_{\J}$ defines an action of $\Z/2\Z$ on $\oneMod_{\sstmod}(\obstuff{R}_d)$ in the sense of \cite[Definition 4.15.1]{EGNO}. Let us denote the $\Z/2\Z$-equivariantization by $\oneMod_{\sstmod}(\obstuff{R}_d)^{\J}$. It is naturally a $\sstmod$-module category.

\begin{proposition}\label{prop:equi-coh}
If $n=2\ell$, then  $\oneMod_{\sstmod}(\obstuff{R}_d)^{\JD}$ is equivalent to $\oneMod_{\sstmod}(\obstuff{B}_d(\obstuff{A}_{\typeD_{\ell+1}}))$ as $\sstmod$-module categories.
If $n=2\ell+1$, then  $\oneMod_{\sstmod}(\obstuff{R}_d)^{\JT}$ is equivalent to $\oneMod_{\sstmod}(\obstuff{B}_d(\obstuff{A}_{\typeT_{\ell}}))$ as $\sstmod$-module categories.
\end{proposition}
\begin{proof}
Let us explicitly construct such equivalence. Take $\algstuff{M}$ a $\obstuff{B}_d$-module where  $\obstuff{B}_d$ denotes either $\obstuff{B}_d(\obstuff{A}_{\typeD})$ or $\obstuff{B}_d(\obstuff{A}_{\typeT})$ according to the parity of $n$. Restricting the action of $\obstuff{B}_d$ to $\obstuff{R}_d$ on $\algstuff{M}$, gives a morphism of $\obstuff{R}_d$-modules: $u_{\algstuff{M}}: \algstuff{M} \otimes_{\obstuff{R}_d} \obstuff{O}_{\J} \to \algstuff{M}$. This morphism is an isomorphism. Indeed, using the inverse of $\theta_{\J}$, we can consider the composition ${(u_{\algstuff{M}}\otimes\mathrm{id}_{\obstuff{O}_{\J}}) \circ (\mathrm{id}_{\algstuff{M}} \otimes \theta_{\J}^{-1}): \algstuff{M} \otimes_{\obstuff{R}_d} \obstuff{R}_d \to \algstuff{M} \otimes_{\algstuff{R}_d} \obstuff{O}_{\J}}$. Finally, associativity of the $\obstuff{B}_d$-action on $\algstuff{M} \otimes_{\algstuff{R}_d} \obstuff{O}_{\J} \otimes_{\algstuff{R}_d} \obstuff{O}_{\J}$ gives:
$u_{\algstuff{M}} \circ \Phi_{\J}(u_{\algstuff{M}})$ is equal to $\Phi^2_{\J}(\algstuff{M}) \iso \algstuff{M}$. Hence $(\algstuff{M},u_{\algstuff{M}}) \in \oneMod_{\sstmod}(\obstuff{R}_d)^{\J}$.

Conversely, take $(\algstuff{M},u_{\algstuff{M}}) \in \oneMod_{\sstmod}(\obstuff{R}_d)^{\J}$, then we can define a $\obstuff{B}_d=\obstuff{R}_d \oplus \obstuff{O}_{\J}$-action on $\algstuff{M}$ using the composition of  $\algstuff{M} \otimes \obstuff{O}_{\J} \to \algstuff{M} \otimes_{\obstuff{R}_d} \obstuff{O}_{\J}$ with $u_{\algstuff{M}}$.
\end{proof}

\begin{remark}
Thanks to \fullref{prop:equi-coh}, we obtain that $\oneMod_{\sstmod}(\obstuff{B}_d(\obstuff{A}_{\typeT_{\ell}}))$ is an exact module category even when the characteristic of $\K$ is equal to $2$. This completes \fullref{prop:type-T-eir}. Indeed, it is enough to show that $\oneMod_{\sstmod}(\obstuff{R}_d)^{\JT}$ is semisimple. We know that $\oneMod_{\sstmod}(\obstuff{R}_d)$ is semisimple since $k=\frac{n}{d}$ is invertible in $\K$ hence $\obstuff{R}_d$ is separable. Furthermore, the action of $\Z/2\Z$ defined by $\Phi_{\JT}$ is free on the simple objects of $\oneMod_{\sstmod}(\obstuff{R}_d)$ since $n$ is odd. Hence thanks to \cite[Remark 4.15.8]{EGNO}, we can conclude that $\oneMod_{\sstmod}(\obstuff{R}_d)^{\JT}$ is also semisimple.
\end{remark}

\begin{proposition}\label{prop:decat-coherence}
For each pair $(\Gamma,d) \in \typeADET(n)$, we have:
\begin{equation*}
\Gros\Big(\oneMod_{\sstmod}\big(\obstuff{B}_d(\obstuff{A}_{\Gamma})\big)\Big) \cong_{\N} \taftm(\Gamma,d).
\end{equation*}
\end{proposition}
\begin{proof}
When $\Gamma$ is of type $\typeADE$, the results of \cite{KO02} identify the simple $\obstuff{A}_{\Gamma}$-modules with the vertices of $\Gamma$. The algebra object $\obstuff{B}_d(\obstuff{A}_{\Gamma})$ adds a cyclic label in $\Z/d\Z$. For type $\typeT$, \fullref{prop:equi-coh} identifies the corresponding module category with the equivariantization of the type $\typeA$ module category. Thus its simple objects are indexed by the orbits of the involution induced by $\obstuff{J}_T$. Now, it is enough to compute the action of $\tsimple[1]$ and $\tinde[2,0]$ on this set of simple objects type by type. That we get the same action as for $\taftm(\Gamma,d)$, follows from \fullref{proposition:taft-tensor}.
\end{proof}

\begin{remark}
Given the underlying $\taft$-module of a connected, simple, separable algebra object whose associated module category is regular, the function \texttt{underlying\_diagram} available here: \href{https://github.com/RaphaelPaegelow/decat\_taft}{semisimple\_taft.jl} computes the pairs $(\Gamma,d)$ associated to it. Note that while $\Gamma$ depends on the algebra structure, $d$ is only defined by this underlying module.
\end{remark}

\subsection{Categorical Clifford theory}\label{subsection:cat-clifford-theory}

In this subsection, let us introduce the portion of categorical Clifford theory needed below. This approach is parallel to the main results of \cite{MM12} but these cannot be directly applied since we can only apply them when $\K$ is of characteristic zero. Since we want to apply the results of this subsection to our concrete situation, let us work in the following framework. Let $\catstuff{C}$ be a fusion category faithfully graded by a finite cyclic group $G$, and let us denote by $\catstuff{D}:=\catstuff{C}_0$, where $0$ is the neutral element of $G$.

 Thanks to \cite[Theorem 6.1]{ENO10}, we have that $\catstuff{C}_g$ is an invertible $\catstuff{D}$-bimodule category for each $g \in G$. Moreover, $m_{g,h}:\catstuff{C}_g \boxtimes_{\catstuff{D}} \catstuff{C}_h \to \catstuff{C}_{gh}$ induced by the tensor product on $\catstuff{C}$ is an equivalence of $\catstuff{D}$-bimodule categories. In particular, $m_{g,g^{-1}}:\catstuff{C}_g \boxtimes_{\catstuff{D}} \catstuff{C}_{g^{-1}} \to \catstuff{D}$ is an equivalence of $\catstuff{D}$-bimodule categories.
%We can then consider a well defined action of $G$ on the set of equivalence classes of exact, indecomposable $\catstuff{D}$-module categories as $g\star [\catstuff{N}] := [\catstuff{C}_g \boxtimes_{\catstuff{D}} \catstuff{N}]$.

\begin{remark}
Although \cite[Theorem 6.1]{ENO10} is stated over an algebraically closed field of characteristic $0$, the proof can be without any effort generalized to an arbitrary algebraically closed field. Indeed, invertibility of the homogeneous components of faithfully graded fusion categories holds in general \cite[Proposition 3.9]{SS26}. Moreover, the proof of \cite[Theorem 6.1 (ii)]{ENO10} applies without a characteristic-zero assumption since it only uses rigidity, semisimplicity and Morita theory.
\end{remark}

Let $\catstuff{M}$ be an exact, indecomposable $\catstuff{C}$-module category. Consider the decomposition
\begin{equation*}
\catstuff{M}_{|\catstuff{D}}= \bigoplus_{\omega \in \Omega}{\catstuff{N}_{\omega}}
\end{equation*}
into indecomposable $\catstuff{D}$-module categories. We then define a group action of $G$ on $\Omega$ in the following way. The action bifunctor of $\catstuff{C}$ on $\catstuff{M}$ induces for each $g\in G$ a natural functor $a_g: \catstuff{C}_g \boxtimes_{\catstuff{D}} \catstuff{M} \to \catstuff{M}$ which is a $\catstuff{D}$-module equivalence since $\catstuff{C}_g$ is an invertible $\catstuff{D}$-bimodule. For each $\omega \in \Omega$, the restriction $a_{g,\omega}$ of $a_g$ to $\catstuff{C}_g \boxtimes_{\catstuff{D}} \catstuff{N}_{\omega}$ is an equivalence onto a unique indecomposable summand of  $\catstuff{M}_{|\catstuff{D}}$. Define $g.\omega$ to be the index of this indecomposable summand. The associativity constraint  of the $\catstuff{C}$-action restricted to $\catstuff{C}_g \times \catstuff{C}_h \times \catstuff{N}_{\omega}$ assures that $g.(h.\omega)=(gh).\omega$, hence that the $G$-action on $\Omega$ is well defined.

\begin{proposition}\label{prop:cat-cliff-1}
\begin{enumerate}[label=(\roman*)]
\item The action of  $G$  on $\Omega$ is transitive.
\item If $\omega \in \Omega$, let $H:=\mathrm{Stab}_{G}(\omega)$, then the $\catstuff{D}$-module structure on $\catstuff{N}_{\omega}$ extends to a $\catstuff{C}_H:=\bigoplus_{h \in H}{\catstuff{C}_h}$-module structure $\tilde{\catstuff{N}}_{\omega}$.
\item For each $\omega \in \Omega$, we have an equivalence of $\catstuff{C}$-module categories:
\begin{equation*}
\catstuff{M} \simeq \catstuff{C} \boxtimes_{\catstuff{C}_{H}} \tilde{\catstuff{N}}_{\omega}
\end{equation*}
where $H:=\mathrm{Stab}_{G}(\omega)$.
\end{enumerate}
\end{proposition}
\begin{proof}
\begin{enumerate}[label=(\roman*)]
\item Let $O \subset \Omega$ be a $G$-orbit and denote by $\catstuff{M}_{O}:=\bigoplus_{\omega \in O}{\catstuff{N}_{\omega}}$. For every $g \in G$, $\catstuff{C}_g \otimes \catstuff{M}_{O} \subset \catstuff{M}_{O}$. Hence $\catstuff{M}_O$ is a $\catstuff{C}$-module subcategory of $\catstuff{M}$. Because $\catstuff{M}$ is indecomposable, this forces the $G$-action to have a unique orbit, hence to be transitive.
\item For every $h \in H$, the essential image of $a_{h,\omega}:\catstuff{C}_h \boxtimes_{\catstuff{D}} \catstuff{N}_{\omega} \to \catstuff{M}$ is $\catstuff{N}_{\omega}$. Hence $a_{h,\omega}:\catstuff{C}_h \boxtimes_{\catstuff{D}} \catstuff{N}_{\omega} \to \catstuff{N}_{\omega}$ is an equivalence. Taking the direct sum over all $h \in H$ gives a functor $\catstuff{C}_H \boxtimes_{\catstuff{D}} \catstuff{N}_{\omega} \to \catstuff{N}_{\omega}$. We can then naturally define an action of $\catstuff{C}_H$ on $\catstuff{N}_{\omega}$ by restricting associativity and unit constraints of the $\catstuff{C}$-action on $\catstuff{M}$.
\item  Note that the restriction of the action functor $\catstuff{N} \boxtimes \catstuff{N}_{\omega} \to \catstuff{M}$ is $\catstuff{C}_H$-balanced and gives a $\catstuff{C}$-module functor $\mathrm{A}:\catstuff{C} \boxtimes_{\catstuff{C}_H} \tilde{\catstuff{N}}_{\omega} \to \catstuff{M}$. Choose now representatives $t \in G$ of the cosets $G/H$. We then have:
\begin{equation*}
\catstuff{C} = \bigoplus_{t \in G/H}{\bigoplus_{h \in H}{\catstuff{C}_{th}}}.
\end{equation*}
If we denote by $\catstuff{C}_{tH}:=\bigoplus_{h \in H}{\catstuff{C}_{th}}$, then $\catstuff{C}_{tH} \boxtimes_{\catstuff{C}_H} \tilde{\catstuff{N}}_{\omega} \simeq \catstuff{C}_t \boxtimes_{\catstuff{D}} \catstuff{N}_{\omega}$ since ${\catstuff{C}_t \boxtimes_{\catstuff{D}} \catstuff{C}_{H} \simeq \catstuff{C}_{tH}}$. This holds because the multiplication functor $m_{t,h}$ is an equivalence for every $h \in H$. Hence we have $\left(\catstuff{C} \boxtimes_{\catstuff{C}_H} \tilde{\catstuff{N}}_{\omega}\right)_{|\catstuff{D}} \simeq \bigoplus_{t \in G/H}{\catstuff{C}_t \boxtimes_{\catstuff{D}} \catstuff{N}_{\omega}}$. By construction, we then have that the restriction of $\mathrm{A}:\catstuff{C}_t \boxtimes_{\catstuff{D}} \catstuff{N}_{\omega} \to \catstuff{N}_{t.\omega}$ is an equivalence. By $(i)$, the categories $\catstuff{N}_{t.\omega}$ for $t \in G/H$ are precisely all the indecomposable $\catstuff{D}$-summands of $\catstuff{M}_{|\catstuff{D}}$. Hence $\mathrm{A}$ is an equivalence.
\end{enumerate}
\end{proof}

\begin{remark}
The \fullref{prop:cat-cliff-1} gives that $|\Omega|=[G:H]$. Since $G$ is a cyclic group, $|\Omega|$ completely determines $H$.
\end{remark}

If $H$ is a subgroup of $G$ and $\tilde{\catstuff{N}}$ a $\catstuff{C}_H$-extension of an exact, indecomposable $\catstuff{D}$-module category $\catstuff{N}$. We can define an action of $G$ on the set of equivalence classes of $g.\tilde{\catstuff{N}}$ as the extension of $\catstuff{C}_g \boxtimes_{\catstuff{D}} \catstuff{N}$ inherited by $\catstuff{C} \boxtimes_{\catstuff{C}_H} \tilde{\catstuff{N}}$, for $g \in G$.

\begin{proposition}\label{prop:cat-cliff-2}
Let $H_1$ and $H_2$ be two subgroups of $G$. Take respectively $\tilde{\catstuff{N}}_1$ and $\tilde{\catstuff{N}}_2$ two $\catstuff{C}_{H_1}$ and $\catstuff{C}_{H_2}$-module categories such that their respective restrictions to $\catstuff{D}$, $\catstuff{N}_1$, $\catstuff{N}_2$ are exact and indecomposable. Then
\begin{equation*}
\catstuff{C}\boxtimes_{\catstuff{C}_{H_1}} \tilde{\catstuff{N}}_1 \simeq \catstuff{C}\boxtimes_{\catstuff{C}_{H_2}} \tilde{\catstuff{N}}_2 \iff \begin{cases} H_1=H_2,\\ \exists g \in G, g.\tilde{\catstuff{N}}_1\simeq \tilde{\catstuff{N}}_2 \text{ as } \catstuff{C}_H\text{-module categories}. \end{cases}
\end{equation*}
\end{proposition}
\begin{proof}
Let $\catstuff{M}_1 := \catstuff{C}\boxtimes_{\catstuff{C}_{H_1}} \tilde{\catstuff{N}}_1$ and $\catstuff{M}_2:=\catstuff{C}\boxtimes_{\catstuff{C}_{H_2}} \tilde{\catstuff{N}}_2$ and suppose that we have $\mathrm{B}:\catstuff{M}_1\to \catstuff{M}_2$ an equivalence of $\catstuff{C}$-module categories. For $i=1,2$, restricting to $\catstuff{D}$ admits the following decomposition in indecomposable $\catstuff{D}$-module categories:
\begin{equation*}
{\catstuff{M}_i}_{|\catstuff{D}} = \bigoplus_{tH_i \in G/H_i}{\catstuff{C}_{tH_i} \boxtimes_{\catstuff{C}_{H_i}} \tilde{\catstuff{N}}_i}.
\end{equation*}
In order to match with the notation used in \fullref{prop:cat-cliff-1} with $\Omega_i:=G/H_i$ and ${\catstuff{N}_{\omega_i}:=\catstuff{C}_{\omega_i} \boxtimes_{\catstuff{C}_{H_i}} \tilde{\catstuff{N}}_i}$ for $\omega_i \in \Omega_i$. The $G$-action is just $g.(tH)=gtH$.
Now, since $\mathrm{B}$ is in particular a $\catstuff{D}$-module equivalence, it induces a $G$-equivariant bijection $\overline{\mathrm{B}}: \Omega_1 \iso \Omega_2$. Therefore, we have $g \in G$ such that $\overline{\mathrm{B}}^{-1}(H_2)=gH_1$. Since $\overline{\mathrm{B}}$ is $G$-equivariant and $G$ is abelian, we get $H_1=H_2$. Let us denote this subgroup by $H$. Finally, because  $\mathrm{B}:\catstuff{N}_{gH} \to \catstuff{N}_{H}$ is compatible with the action of $\catstuff{C}_h$ for every $h \in H$, we can consider the $\catstuff{C}_H$-extensions which by definition gives that $g.\catstuff{N}_1 \simeq \catstuff{N}_2$ as $\catstuff{C}_H$-module categories. The implication in the opposite direction follows directly from \fullref{prop:cat-cliff-1}.
\end{proof}

If $\catstuff{N}$ is an exact, indecomposable $\catstuff{D}$-module category, let
\begin{center}
$\pi_{\catstuff{N}}:\begin{array}{ccc}
\mathrm{Aut}_{\catstuff{D}}(\catstuff{N}) &\to& \mathrm{Sym}\left(\mathrm{Irr}(\catstuff{N})\right)\\
 \mathrm{B} &\mapsto & ([S] \mapsto [\mathrm{B}(S)])
\end{array}$
\end{center}
where $\mathrm{Irr}(\catstuff{N})$ denotes the set of isomorphism classes of simple objects of the semisimple category $\catstuff{N}$.

\begin{proposition}\label{prop:cat-cliff-3}
Let $H=\langle  h \rangle$ be a subgroup of $G$ and $\obstuff{U} \in \catstuff{C}_h$ be an invertible object. Let $\tilde{\catstuff{N}}_1$, $\tilde{\catstuff{N}}_2$ be two $\catstuff{C}_H$-extensions of the same underlying exact, indecomposable $\catstuff{D}$-module category $\catstuff{N}$. Denote by $\otimes_1$ and $\otimes_2$ the respective $\catstuff{C}_H$-actions.\\
If $\pi_{\catstuff{N}}$ is injective and $\obstuff{U} \otimes_1 \obstuff{S} \simeq \obstuff{U} \otimes_2 \obstuff{S}$ for every simple object of $\catstuff{N}$, then $\tilde{\catstuff{N}}_1 \simeq \tilde{\catstuff{N}}_2$ as $\catstuff{C}_H$-module categories.
\end{proposition}
\begin{proof}
Let us denote by $\psi_i: \catstuff{C}_h \boxtimes_{\catstuff{D}} \catstuff{N} \to \catstuff{N}$ be the equivalence induced by the extension $\tilde{\catstuff{N}}_i$, for $i=1,2$. Moreover, $\obstuff{U}$ induces two autoequivalences $\mathrm{T}_i$ of $\catstuff{N}$ defined as ${\mathrm{T}_i(\obstuff{N}):=\psi_i(\obstuff{U} \boxtimes \obstuff{N})=\obstuff{U}\otimes_i \obstuff{N}}$. Let $\mathrm{Q}$ denote the $\catstuff{D}$-module autoequivalence $\psi_2 \circ \psi_1^{-1}$, where $\psi_1^{-1}$ is one quasi-inverse of $\psi_1$. By hypothesis, $\mathrm{Q}(\mathrm{T}_1(\obstuff{S})) \simeq \mathrm{T}_2(\obstuff{S}) \simeq \mathrm{T}_1(\obstuff{S})$, for every simple object $\obstuff{S}$ of $\catstuff{N}$. Hence, $\pi_{\catstuff{N}}([\mathrm{Q}])= \mathrm{id}$. By injectivity of $\pi_{\catstuff{N}}$, we have that $\mathrm{Q} \simeq \mathrm{Id}_{\catstuff{N}}$. This implies that $\psi_1 \simeq \psi_2$ as $\catstuff{D}$-module equivalences. It remains to check that the associated associativity constraint are equal. Indeed this is the case because $\catstuff{N}$ is indecomposable, hence they might differ by an element of $\mathrm{H}^2(H,\K^*)$. Since $H$ is cyclic and $\K$ is algebraically closed, $\mathrm{H}^2(H,\K^{\ast})=\K^{\ast}/(\K^{\ast})^{\mathrm{ord}(h)}=0$, hence $\tilde{\catstuff{N}}_1 \simeq \tilde{\catstuff{N}}_2$ as $\catstuff{C}_H$-module categories.
\end{proof}

\subsection{Classification of exact, indecomposable, regular module categories of $\sstmod$}\label{subsection:eir-mod-cat-sst}

\begin{lemma}\label{lemma:application-cat-cliff-3}
If $\catstuff{M}$ be a $\sstmod$-module category such that $\Gros(\catstuff{M}) \cong_{\N} \taftm(\Gamma,d)$ for a certain pair $(\Gamma,d) \in \typeADET(n)$, then there exists an indecomposable $\catstuff{W}$-module category $\catstuff{N}$ which is a summand of $\catstuff{M}_{|\catstuff{W}}$ such that:
\begin{enumerate}[label=(\roman*)]

\item
\begin{equation*}
\mathrm{Stab}_{C_{2n}}(\catstuff{N}) = \begin{cases} \langle 2d \rangle & \text{ if }\Gamma \text{ is of type } \typeADE, \\ \langle d \rangle &\text{ if } \Gamma \text{ is of type } \typeT. \end{cases}
\end{equation*}
\item The equivalence class of $\catstuff{N}$ as a $\catstuff{W}$-module category is independent of $\catstuff{M}$.
\item $\pi_{\catstuff{N}}$ is injective.
\end{enumerate}
\end{lemma}
\begin{proof}
When $n=2,3$ we have $\catstuff{W} \simeq \mathrm{Vec}$. Suppose now that $n \geq 4$.
Let us first construct $\catstuff{N}$ and then check that it has all the claimed properties. Choose $(\obstuff{L}_{ij})_{(i,j) \in \llbracket 1, v(\Gamma) \rrbracket \times \Z/d\Z}$ simple objects of $\catstuff{M}$ such that $[\obstuff{L}_{ij}]$ corresponds to $m_{ij}\in \taftm(\Gamma,d)$.\\
$\bullet$ If $\Gamma$ is of type $\typeADE$, the underlying graph is a connected tree. It thus admits a unique bipartition (up to interchanging the two colors) $\Gamma= \Gamma_0 \coprod \Gamma_1$. Let us denote by $\Aut^+(\Gamma)$ the group of graph automorphisms that preserves the two subsets of colors $\Gamma_0$ and $\Gamma_1$ and if $X$ is a set, let $\Sym(X)$ denote the symmetric group of $X$. We claim that there exists $\epsilon \in \{0,1\}$ such that the natural map $\Aut^+(\Gamma) \to \Sym(\Gamma_{\epsilon})$ is injective. Let us define a function $h:\llbracket 1, v(\Gamma) \rrbracket \to \Z_{\geq 0}$ such that $1 \in \Gamma_{h(1)}$ and for each oriented edge $x \to y$, $h(y)=h(x)+1$. In particular, we get $\left \{i \in \llbracket 1, v(\Gamma) \rrbracket \mid h(i) \equiv s \mod 2 \right\}=\Gamma_s$ for $s \in \{0,1\}$. Consider $\catstuff{N}:=\langle \{L_{ij} \mid 2j+h(i) \equiv \epsilon \mod 2d\} \rangle \subset \catstuff{M}$.\\
$\bullet$ If $\Gamma$ is of type $\typeT$, choose $a \in \Z/d\Z$ and consider $\catstuff{N}:=\langle \{L_{ij} \mid 2j+i \equiv a \mod d\} \rangle \subset \catstuff{M}$.

As defined $\catstuff{N}$ is clearly a $\catstuff{W}$-module category since $\catstuff{W}$ is tensor generated by $\obstuff{W}_1$ and $[\obstuff{W}_1]=\ypar^{\smin 1}\xpar^2-[\tsimple[0]]$. Moreover, the action of $\obstuff{W}_1$ is connected on the simples of $\catstuff{N}$, hence $\catstuff{N}$ is indecomposable. The fact that the stabilizer of $\catstuff{N}$ in the grading group $C_{2n}$ is the correct subgroup follows directly from the construction. Recall from \fullref{subsec:stable-cat}, that $\catstuff{C}_{n-2}(\xi)$ denotes the Verlinde category associated with $(\mathfrak{sl}_2,\xi)$, where  $\xi^2=\qpar\in \K$. Take $\mathrm{B} \in \Aut_{\catstuff{W}}(\catstuff{N})$ such that $\pi_{\catstuff{N}}(\mathrm{B})=\mathrm{id}$ and denote by $\hat{\catstuff{N}}:= \catstuff{C}_{n-2}(\xi) \boxtimes_{\catstuff{W}} \catstuff{N}$. Let us argue $(ii)$ and  $(iii)$ by splitting type $\typeADE$ and type $\typeT$.\\
Let $\Gamma$ be of $\typeADE$. Take two $\sstmod$-module categories $\catstuff{M}_1$ and $\catstuff{M}_2$ such that $\Gros(\catstuff{M}_i) \cong_{\N} \taftm(\Gamma,d)$. Let us show that $\catstuff{N}_1 \simeq \catstuff{N}_2$ as $\catstuff{W}$-module categories to obtain $(ii)$ in that case. The $\obstuff{V}_1$-fusion graph of $\hat{\catstuff{N}}_i$, the graph describing the action of $[\obstuff{V}_1]$ on $\Gros(\hat{\catstuff{N}}_i)$ with respect to the basis of simple objects of $\hat{\catstuff{N}}_i$ cf. \cite[\S 3.1]{EO04-2}, is $\Gamma$. By the classification of $\typeADE$ modules categories from \fullref{subsection:explicit-const-mod-cat}, we obtain an equivalence $\mathrm{G}:\hat{\catstuff{N}}_1 \to \hat{\catstuff{N}}_2$. Recall that $\catstuff{C}_{n-2}(\xi)$ is $\Z/2\Z$ faithfully graded with neutral component $\catstuff{W}$. The simple objects of the two indecomposable $\catstuff{W}$-summands of $\hat{\catstuff{N}}_i$ are indexed by $\Gamma_0$ and $\Gamma_1$, respectively. The equivalence $\mathrm{G}$ induces an automorphism of $\Gamma$. Since $\Gamma$ is connected and bipartite, a graph automorphism either preserves $\Gamma_0$ and $\Gamma_1$ or exchanges them. In the first case, this gives an equivalence between $\catstuff{N}_1$ and $\catstuff{N}_2$. In the second case, note that the only $\typeADE$ graph admitting an automorphism exchanging the two colors is even type $\typeA$. In that case $n=h(\Gamma)$ is odd and $\hat{\catstuff{N}}_i$ is equivalent to the regular $\catstuff{C}_{n-2}(\xi)$-module category. Composing $\mathrm{G}$ with $\_\otimes \obstuff{V}_{n-2}$ gives an equivalence preserving the bipartition, since $\obstuff{V}_{n-2}$ is invertible. After restricting to $\catstuff{W}$, we have constructed an equivalence between $\catstuff{N}_1$ and $\catstuff{N}_2$.\\
For $(iii)$, the induced autoequivalence $\hat{\mathrm{B}}:=\mathrm{Id}_{\catstuff{C}_{n-2}(\xi)} \boxtimes \mathrm{B}$ defines a graph automorphism of $\Gamma$ preserving $\Gamma_0$ and $\Gamma_1$, hence an element of $\Aut^+(\Gamma)$. By the choice of $\epsilon$, we get that $\hat{\mathrm{B}}$ fixes every simple object of $\hat{\catstuff{N}}$. Up to some rescaling, we get $\hat{\mathrm{B}} \simeq \mathrm{Id}_{\hat{\catstuff{N}}}$, which gives the desired result after restricting to $\catstuff{N}$.

Finally, in type $\typeT$, it is simpler since $\catstuff{W} \to \catstuff{N}, \obstuff{X} \mapsto \obstuff{X} \otimes \obstuff{L}_{1j_1}$ where $j_1$ is the unique element of $\Z/d\Z$ such that $2j_1+1 \equiv a \mod d$ defines an equivalence of $\catstuff{W}$-module categories. This directly gives $(ii)$. For $(iii)$, after this identification, we get that $\mathrm{B}$ is necessary of the form $\mathrm{B}(\obstuff{X})\simeq \obstuff{X} \otimes \mathrm{B}(\tsimple[0])$. Since $\pi_{\catstuff{N}}(\mathrm{B})=\mathrm{id}$, we get in particular that $\mathrm{B}(\tsimple[0]) \simeq \tsimple[0]$, hence that $\mathrm{B}\simeq \mathrm{Id}_{\catstuff{N}}$.
\end{proof}

We are now in position to prove a classification of all regular exact indecomposable module categories of $\sstmod$.

\begin{theorem}\label{thm:class-bottom}
Taking the split Grothendieck group induces a bijection between
\begin{equation*}
\left\{ \begin{gathered} \text{ exact, indecomposable, regular }\\ \sstmod\text{-module categories} \end{gathered} \right\} \big/_{\simeq} \iso \left\{ \begin{gathered} \text{admissible, regular, transitive, NIM} \\  \N\text{-modules with apex }\cell[b] \end{gathered} \right\} \big/_{\cong_{\N}}.
\end{equation*}
\end{theorem}
\begin{proof}
The fact that this map is well defined follows from \fullref{prop:decat-is-ADET}. Moreover, surjectivity follows from the explicit constructions done in \fullref{subsection:explicit-const-mod-cat}.

Suppose now that $\Gros(\catstuff{M}_1)\cong_{\N} \Gros(\catstuff{M}_2)$. Thanks to \fullref{proposition:taft-small-apex}, we have $(\Gamma,d)$ such that $\Gros(\catstuff{M}_1) \cong_{\N} \taftm(\Gamma,d)$. Thanks to \fullref{lemma:application-cat-cliff-3}, one can choose two indecomposable summands $\catstuff{N}_i$ of ${\catstuff{M}_i}_{|\catstuff{W}}$ such that $\catstuff{N}_1 \simeq \catstuff{N}_2$ as $\catstuff{W}$-module categories and having the same stabilizer $H$ inside $C_{2n}$. Take
\begin{equation*}
\obstuff{U} := \begin{cases}\tsimple[d] & \text{ if } \Gamma \text{ of type } \typeADE,\\ \obstuff{K}_d & \text{ if } \Gamma \text{ of type } \typeT. \end{cases}
\end{equation*}
From the decategorification we then have $\obstuff{U} \otimes_{i} \obstuff{S} \simeq \obstuff{S}$. Hence \fullref{prop:cat-cliff-3} gives that $\tilde{\catstuff{N}}_1 \simeq \tilde{\catstuff{N}}_2$ as $\sstmod_{H}$-module categories. Finally \fullref{prop:cat-cliff-2} gives that $\catstuff{M}_1 \simeq \catstuff{M}_2$ and thus the injectivity of the split Grothendieck group.
\end{proof}

\begin{remark}
Pushing further the categorical Clifford theory should give that nonregular exact, indecomposable module categories of the semisimplification of the Taft category are classified by $(\Gamma,d, [\sigma])$ where $\sigma$ is a conjugacy class of an automorphism of $\Gamma$ preserving the bipartition of $\Gamma$ (if it exists) and of order dividing $k=n/d$.
For instance the admissible, transitive, NIM $\N$-module given in \fullref{ex:non-reg} is associated to the classifying data $(\typeA_3, 1, [s])$ where $s$ is the reflection that exchanges the two outer vertices of $\typeA_3$.
\end{remark}

\section{Cell \texorpdfstring{$2$}{2}-representations of \texorpdfstring{$\tmod$}{$\tmod$}}
In this last section, we will give explicitly expressions of the fundamental cell $2$-representations of $\tmod$. Recall from \fullref{rmk:cells} that in this context, the notions of left, right and two-sided cell coincide.

\begin{definition}\label{definition:2-cell-module}
Take $\cell[]$ a cell of $\tmod$. Let $\twocatstuff{L}_{\cell[]}(\geq)$ be the $2$-representation
of $\tmod$ spanned by
the additive closure of all indecomposable objects $\obstuff{X}$
which belong to the union of all cells $\cell[]^{\prime}\geq\cell[]$. Moreover, let $\catstuff{Z}(\geq)$ be the unique, maximal proper $\tmod$-stable ideal of $\twocatstuff{L}_{\cell[]}(\geq)$ that does not contain the identity morphism of any indecomposable object of $\cell[]$ cf. \cite[Lemma 3]{MM16c}.
We call $\twocatstuff{L}_{\cell[]}=\twocatstuff{L}_{\cell[]}(\geq)/\catstuff{Z}(\geq)$
the \textit{cell $2$-representation of $\cell[]$}.
\end{definition}

Recall the cells of $\tmod$ as calculated in \fullref{lemma:taft-cells},
and let us denote the category of finite-dimensional, projective $\taft$-modules by $\tpmod$.
By \fullref{lemma:taft-simples}, the indecomposable objects of $\tpmod$ are
the projective covers $\tprmod$ of the simples $\tsimple$.
\begin{proposition}\label{proposition:taft-cell-cat}
The cell $2$-representation $\twocatstuff{L}_{\cell[u]}$
is equivalent to $\tpmod$, as a $2$-representation of $\tmod$.
\end{proposition}
\begin{proof}
Thanks to \fullref{lemma:taft-cells}, it is enough to prove that $\tpmod$ does not contain
any nonzero, proper $\tmod$-stable ideals.

Consider any nonzero morphism of $\taft$-modules $f\colon\tprmod[1]\to\tprmod[0]$. We are
going to show that the $\tmod$-stable ideal generated by $f$ contains the identity morphism of
$\tprmod[0]$ and is therefore not proper.
By \fullref{lemma:decomp-A}, we have
\begin{gather}\label{eq:decomp-aob}
\morstuff{U}(\taobj[1,0])\cong{\textstyle\bigoplus_{i=0}^{n\smin 1}}\,\tprmod[1\smin i]
\end{gather}
in $\tpmod$. As we showed in the proof of that lemma, under the isomorphism in
\eqref{eq:decomp-aob}, we have $\taobj[1,0](i)\cong\tprmod[1{-}i]$.
Therefore, we can take
$\{\gat(\xx^{j}\acts\zapar^{n{-}1})\mid j=0,\dots,n-1\}$ as a basis
of $\taobj[1,0](1)\cong\tprmod[0]$ and
$\{\xx^{j}\acts\zapar^{n{-}1}\mid j=0,\dots,n-1\}$ as a basis of
$\taobj[1,0](0)\cong\tprmod[1]$. On these bases,
the morphism $f\colon\tprmod[1]\to\tprmod[0]$ is uniquely determined by
\[
f(\zapar^{n{-}1})=c\gat(\xx\acts\zapar^{n{-}1}),
\]
for a certain nonzero scalar $c\in\K$. Without loss of generality, we can take $c=-\qpar^{-2}$.

There is a nonzero morphism of $\taft$-modules
$\iota\colon\tprmod[0]\to\tprmod[1]\otimes\tprmod[1]$,
uniquely determined by $\iota(\gat\zapar^{n{-}1})=\zapar\otimes\zapar^{n{-}1}$.
Similarly, there is a nonzero
morphism of $\taft$-modules $\pi\colon\tprmod[1]\otimes\tprmod[0]\to\tprmod[0]$, uniquely
determined by embedding $\tprmod[1]\otimes\tprmod[0]$
into $\taobj[1,0]\otimes\taobj[1,0]$,
multiplying in $\taobj[1,0]$ and projecting onto $\tprmod[0]$, using
\eqref{eq:decomp-aob}.

A simple calculation now shows that
\[
\pi(\mathrm{id}_{\tprmod[1]}\otimes f)\iota=\mathrm{id}_{\tprmod[0]}.
\]
More precisely,
\begin{gather*}
\pi(\mathrm{id}_{\tprmod[1]}\otimes f)\iota(\gat\zapar^{n-1})=-\qpar^{-2}\zapar\gat(\xx\acts\zapar^{n-1})
=-\qpar^{-2}\zapar\gat(1+\qpar^{{-}1}+\dots+\qpar^{2{-}n})\zapar^{n{-}2}
=\gat\zapar^{n{-}1},
\end{gather*}
using \eqref{eq:x-action} and that $1+\qpar^{-1}+\cdots+\qpar^{2{-}n}=-\qpar$.

If $f\colon\tprmod[i]\to\tprmod[j] \neq 0$ lies in a stable ideal $\algstuff{I}$, then $\mathrm{id}_{\taft} \otimes f \in \algstuff{I}$. By \cite[Proposition 6.1.11]{EGNO}, we get $\mathrm{id}_{\taft} \in \algstuff{I}$. Since $\taft \simeq \bigoplus_{r=0}^{n\smin 1}{\tprmod[r]}$ in $\tmod$, we have $\mathrm{id}_{\tprmod[r]} \in \algstuff{I}$, hence $\algstuff{I}=\tpmod$.
%Similar arguments prove the same for nonzero morphisms
%$f\colon\tprmod[i]\to\tprmod[j]$ in general,
%so there are indeed no nonzero proper $\tmod$-stable ideals in $\tpmod$.
\end{proof}

\begin{remark}
Even though $\tmod$
is a tensor category, $\twocatstuff{L}_{\cell[u]}$ is an example of a nonabelian cell $2$-representation, since $\tpmod$ is not abelian thanks to \fullref{subsection:taft-reps}.
\end{remark}

\begin{proposition}\label{proposition:taft-decat-prmod}
We have $\Gros(\twocatstuff{L}_{\cell[u]})\cong_{\N}\taftm(n)$ as $\taftp$-modules.
\end{proposition}
\begin{proof}
This follows directly from \fullref{lemma:taft-cells}
and \fullref{proposition:taft-cell-cat}.
\end{proof}

The following proposition shows that the failure of decategorification to distinguish upper-cell $2$-representations is not caused by having different collections of objects. Even if $\oneMod_{\tmod}(\taobj[n,\lambda])$ and $\tpmod \equiv \twocatstuff{L}_{\cell[u]}$ have essentially the same indecomposable objects and the same decategorification, the distinction lies in their morphism spaces.

\begin{proposition}\label{proposition:taft-2-rep}
For any $\lambda\in\K$, the forgetful functor
\[
\morstuff{U}\colon\oneMod_{\tmod}(\taobj[n,\lambda])\to\tpmod
\]
is an essentially surjective, faithful, but non-full, morphism of $2$-representations of $\tmod$.
\end{proposition}
\begin{proof}
By \fullref{lemma:decomp-A}, we have
$\morstuff{U}(\taobj[n,\lambda])\cong\tprmod[1]$.
Since $\tsimple\otimes\tprmod[1]\cong\tprmod[1{+}r]$, for any $r=0,\dots,n\smin 2$,
we can use \fullref{theorem:semisimplicity} to
see that the indecomposable objects in
$\oneMod_{\tmod}(\taobj[n,\lambda])$ are precisely
the indecomposable projective $\taft$-modules. However, it is
not equivalent to $\tpmod$, because $\oneMod_{\tmod}(\taobj[n,\lambda])$ is
semisimple by \fullref{corollary:taft-module-cat}.
\end{proof}

Let us conclude this section by identifying the cell $2$-representation of the bottom cell $\cell[b]$ with the semisimplification of the Taft category.

\begin{proposition}
The cell $2$-representation $\twocatstuff{L}_{\cell[b]}$
is equivalent to $\sstmod$, as $2$-representation of $\tmod$.
\end{proposition}
\begin{proof}
Thanks to \fullref{lemma:taft-cells}, it is enough to prove that $\catstuff{Z}(\geq)=\catstuff{N}$, the tensor ideal of negligible morphisms of $\tmod$. Indeed, $\sstmod$ being a fusion category, it is a simple $2$-representation of $\sstmod$, hence a simple $2$-representation of $\tmod$. Moreover $\catstuff{N}$ does not contain the identify morphism of any object of $\cell[b]$. We thus get the desired equality.
\end{proof}

\printbibliography
\end{document}